\documentclass[reqno,11pt]{amsart}
\usepackage{amsmath,amsxtra,amsbsy,enumerate, latexsym, amsfonts, amssymb, amsthm, amscd, stmaryrd}
\usepackage{hyperref}
\usepackage{graphics,epsf,psfrag}
\usepackage{bbm, dsfont}
\usepackage{mathtools}
\usepackage{xcolor}
\usepackage{float}
\usepackage{tikz}
\usepackage[headings]{fullpage}
\usepackage{colonequals} 
\usepackage{microtype}
\numberwithin{equation}{section}
\allowdisplaybreaks

\definecolor{darkred}{rgb}{0.7,0.1,0.1}

\definecolor{darkgreen}{rgb}{0.1,0.7,0.1}

\newcommand{\bbE}{{\ensuremath{\mathbb E}} }

\newcommand{\bbN}{{\ensuremath{\mathbb N}} }

\newcommand{\bbP}{{\ensuremath{\mathbb P}} }

\newcommand{\bbR}{{\ensuremath{\mathbb R}} }

\newcommand{\bbZ}{{\ensuremath{\mathbb Z}} }

\renewcommand{\epsilon}{\varepsilon}

\newcommand{\gep}{\varepsilon}       

\newcommand{\gG}{\Gamma}

\newcommand{\go}{\omega}

\newcommand{\gl}{\lambda}

\newcommand{\cF}{{\ensuremath{\mathcal F}} }

\newcommand{\cP}{{\ensuremath{\mathcal P}} }

\newcommand{\cU}{{\ensuremath{\mathcal U}} }

\newcommand{\cV}{{\ensuremath{\mathcal V}} }
\DeclareMathOperator*{\supess}{ess\, sup}

\newcommand{\ind}{\mathbf{1}}

\newcommand{\lint}{\llbracket}
\newcommand{\rint}{\rrbracket}
\newcommand{\intp}[1]{\lint #1 \rint}

\newtheorem{theorem}{Theorem}[section]
\newtheorem{lemma}[theorem]{Lemma}
\newtheorem{proposition}[theorem]{Proposition}
\newtheorem{cor}[theorem]{Corollary}
\newtheorem{rem}[theorem]{Remark}

\newtheorem{conjecture}[theorem]{Conjecture}
\newtheorem{thmx}{Theorem}

\newcommand{\RN}[1]{%
  \textup{\uppercase\expandafter{\romannumeral#1}}%
}

\newcommand{\dd}{\mathrm{d}}

\renewcommand{\tilde}{\widetilde}
\renewcommand{\hat}{\widehat}

\newcommand{\union}{\bigcup}
\newcommand{\sumtwo}[2]{\sum_{\substack{#1 \\ #2}}} 
\newcommand{\uniontwo}[2]{\union_{\substack{#1 \\ #2}}} 

\makeatletter
\def\captionfont@{\footnotesize}
\def\captionheadfont@{\scshape}

\long\def\@makecaption#1#2{%
  \vspace{2mm}
  \setbox\@tempboxa\vbox{\color@setgroup
    \advance\hsize-6pc\noindent
    \captionfont@\captionheadfont@#1\@xp\@ifnotempty\@xp
        {\@cdr#2\@nil}{.\captionfont@\upshape\enspace#2}%
    \unskip\kern-6pc\par
    \global\setbox\@ne\lastbox\color@endgroup}%
  \ifhbox\@ne 
    \setbox\@ne\hbox{\unhbox\@ne\unskip\unskip\unpenalty\unkern}%
  \fi
  \ifdim\wd\@tempboxa=\z@ 
    \setbox\@ne\hbox to\columnwidth{\hss\kern-6pc\box\@ne\hss}%
  \else 
    \setbox\@ne\vbox{\unvbox\@tempboxa\parskip\z@skip
        \noindent\unhbox\@ne\advance\hsize-6pc\par}%
\fi
  \ifnum\@tempcnta<64 
    \addvspace\abovecaptionskip
    \moveright 3pc\box\@ne
  \else 
    \moveright 3pc\box\@ne
    \nobreak
    \vskip\belowcaptionskip
  \fi
\relax
}
\makeatother

\def\writefig#1 #2 #3 {\rlap{\kern #1 truecm
\raise #2 truecm \hbox{#3}}}

\newcommand{\p}{{p^*}}
\newcommand{\q}{{q^*}}
\newcommand{\1}{\ind}
\newcommand{\E}{\bbE}
\renewcommand{\P}{\bbP}
\newcommand{\N}{\bbN}
\newcommand{\eps}{\varepsilon}
\renewcommand{\complement}{\mathsf{c}}
\newcommand{\tf}{\textsc{f}}
\renewcommand{\emptyset}{\varnothing}
\newcommand{\argmax}{\operatorname{arg\, max}}

\begin{document}

\author{Stefan Junk}
\address{Gakushuin University, 1-5-1 Mejiro, Toshima-ku, Tokyo 171-8588 Japan}
\email{sjunk@math.gakushuin.ac.jp}
\author{Hubert Lacoin}
\address{IMPA, Estrada Dona Castorina 110, 
Rio de Janeiro RJ-22460-320- Brasil}
\email{lacoin@impa.br}

\title[]{The tail distribution of the partition function for directed polymers in the weak disorder  phase}

\begin{abstract}
We investigate the upper tail distribution of the partition function of the directed polymer in a random environment on $\bbZ^d$ in the weak disorder phase. We show that the distribution of the infinite volume partition function $W^{\beta}_{\infty}$ displays a power-law decay, with an exponent $p^*(\beta)\in [1+\frac{2}{d},\infty)$. We also prove that the distribution of the suprema of the point-to-point and point-to-line partition functions display the same behavior. On the way to these results, we prove a technical estimate of independent interest:
the $L^p$-norm of the partition function at the time when it overshoots a high value $A$ is comparable to $A$.
We use this estimate to extend the validity of many recent results that were  proved under the assumption that the environment is upper bounded.   \\[10pt]
\end{abstract}

\maketitle

\section{Introduction}

The directed polymer in random environment (or DPRE) on $\bbZ^d$ is a model in statistical mechanics involving a random walk (or polymer) interacting with  a disordered medium. It was introduced  \cite{HH85}, for $d=1$, as a simplified model to describe the interfaces of the planar Ising model with random coupling constants at low temperature, and was generalized to higher dimension soon afterwards \cite{IS88}.

\medskip

The DPRE and variations of the model have received much attention from the mathematical community for a wide variety of reasons -- we refer to \cite{Zy24} for a recent survey.  In this introduction, we discuss mainly the model in spatial dimension $d\ge 3$ as this is the object of the present work. In that setup, the DPRE undergoes a phase transition from a high-temperature, weak disorder phase to a low-temperature, strong disorder phase, each phase exhibiting a radically different behavior. In the weak disorder phase,  on large scales, the polymer trajectory is not affected by the disorder and displays the same behavior as a simple random walk. In particular, it converges to a standard Brownian motion under a diffusive scaling (see \cite{CY06} and references therein).

\medskip

On the other hand, in the strong disorder phase, it is conjectured  that there exists a corridor where the environment is particularly favorable and around which the trajectories localize with high probability. It is further predicted that this corridor is \textit{superdiffusive} in the sense that its transversal fluctuations are much larger than $\sqrt{n}$ where $n$ denotes the system's length.
While the rigorous understanding of this  \textit{pathwise} localization phenomenon is still rudimentary and usually requires assumptions beyond strong disorder (see e.g.\ \cite{CC13, bates21} for progresses in that direction), a more precise picture has emerged concerning localization for the \textit{endpoint} of the polymer (see for instance \cite{CH02,CSY03,BC20}).

\medskip

The weak and strong disorder regimes are defined in terms of the asymptotic behavior of the partition function $W^{\beta}_n$ (defined in Equation \eqref{defwn} below). Weak disorder holds if the sequence $(W^{\beta}_n)_{n\ge 1}$ is uniformly integrable and converges to a non-trivial limit $W^{\beta}_{\infty}$, while strong disorder holds if $W^{\beta}_n$ converges to zero.
There exists a critical value of $\beta_c$ which separates the weak and strong disorder phases.
The understanding of the phase transition from weak to strong disorder has been significantly improved in \cite{JL24} by proving (under a technical assumption on the distribution of $\go$) that whenever $W^{\beta}_n$ decays  to zero, it does so exponentially fast, and that consequently, weak disorder holds at $\beta_c$. However, most questions concerning this transition remain widely open. Let us mention two of them:
\begin{itemize}
  \item [(A)] How regular is the free energy curve around $\beta_c$? The free energy, defined as $\mathfrak{f}(\beta)\coloneqq \lim\limits_{n\to\infty} n^{-1}\log W^{\beta}_n$, is equal to zero for $\beta\le \beta_c$ and is negative for $\beta>\beta_c$, hence it is not analytic at $\beta_c$, but currently not much is known beyond this.
\item [(B)]  What features, if any, distinguish the behavior of the system at criticality $(\beta=\beta_c)$ from the interior of the weak disorder phase ($\beta\in[0,\beta_c)$)?
\end{itemize}

\medskip

In the present work, we study the probability distribution of the limit $W^{\beta}_{\infty}$ in the weak disorder phase. We prove that  $\bbP(W^{\beta}_\infty> u)$ decays, up to a multiplicative constant, like $u^{-p^*(\beta)}$ as $u\to\infty$ for an exponent $p^*(\beta)\in [1+\frac{2}{d},\infty)$.
In the process we also obtain comparable bounds for $\sup_{n\ge 1} W_n^\beta$ and for the supremum over all point-to-point partition functions. 

\medskip

Many properties of the function $\beta \mapsto p^*(\beta)$ have been proved in earlier works \cite{J22,J23,J21_2,JL24} (with an appropriate definition for $p^*(\beta)$ see Equation \eqref{defpbeta}) sometimes with additional technical assumptions concerning the distribution on the environment $\go$. More precisely in many instances $\go$ is assumed to be upper bounded. In this paper, we provide alternative proofs for some of these statements which do not require any assumption besides exponential integrability of $\go$.

\medskip

The key estimate that allows us to prove our results with greater generality is
 Proposition~\ref{overshoot} which controls the amount of overshoot of $W_n^\beta$ at the first time it exceeds a large value $A$. This technical result is likely to find further application in the study of directed polymers, and could also be generalized and helpful in the study of other disordered models.

\section{Model and results}\label{sec:model}

\subsection{Definition and previous results}
Let $X=(X_k)_{k\ge 0}$ be the nearest-neighbor simple random walk on $\bbZ^d$ starting from the origin and $P$ its law. We have $P(X_0=0)=1$ and the increments $(X_{k+1}-X_k)_{k\ge 1}$ are independent and identically distributed (i.i.d.) with distribution

\begin{equation}\label{SRWdef}
P(X_{1}=x)=\frac{\ind_{\{|x|=1\}}}{2d}
\end{equation}
where
$|\cdot|$ denotes the $\ell_1$ distance on $\bbZ^d$.
Given a collection $\go=(\go_{k,x})_{k\ge 1, x\in \bbZ^d}$ of real-valued weights (the environment), a parameter $\beta\ge 0$ (the inverse temperature) and $n\ge 1$ (the polymer length), we define the polymer measure $P^{\beta,\go}_{n}$
as a modification of the distribution $P$ which favors trajectories that visit sites where $\go$ is large.
More precisely,
to each path $\pi\colon \bbN\to\bbZ^d$ we associate an energy
	$H_n(\omega,\pi)\coloneqq \sum_{i=1}^n \omega_{i,\pi(i)},$
and we set
\begin{align}\label{eq:mu}
	P_{n}^{\beta,\go}(\dd X)\coloneqq\frac{1}{Z_{n}^\beta}e^{\beta H_n(\omega,X)}P(\dd X),  \quad  \text{ where } \quad Z_{n}^\beta\coloneqq  E\left[  e^{\beta H_n(\omega,X)}\right].
\end{align}
The quantity $Z_{n}^\beta$ is referred to as the partition function of the model (note that $Z_n^\beta$ depends on $\omega$).  Note that the normalization by $Z_{n}^\beta$ makes $P_{n}^{\beta,\go}$ a probability.
In what follows, we assume that the environment $\go$ is given by a fixed realization of an i.i.d.\@ random field on $\N\times\bbZ^d$ with law $\bbP$. We assume that $\go$ has finite exponential moments of all orders, that is
\begin{equation}\label{expomoment}
  \forall \beta\ge 0, \ \gl(\beta)\coloneqq \log \bbE[ e^{\beta\go_{1,0}}]<\infty.
\end{equation}
Note that we do not make any assumptions on the tail of $\omega_{1,0}$ at $-\infty$.
By Fubini, it is easy to check that
$\bbE[Z_{n}^\beta]= e^{n\gl(\beta)}$, and it is thus natural to define the normalized partition function $W_{n}^\beta$ by setting
\begin{equation}\label{defwn}
 W_{n}^\beta= \frac{Z_{n}^\beta}{\bbE[Z_{n}^\beta]}
 =E\left[  e^{\beta H_n(\omega,X)-n\gl(\beta)}\right].
\end{equation}
The normalized partition function  $W_{n}^\beta$ encodes essential information on the typical behavior of $X$ under $P^{\beta,\go}_{n}$ and has been a central object of attention in the study of this model.
An important observation made in \cite{B89} is that $W^{\beta}_n$ is a  martingale with respect to the filtration  $\cF_n\coloneqq \sigma(\omega_{k,x}\colon k\leq n)$, and hence converges almost surely to a limit $W^\beta_\infty\in[0,\infty)$. Furthermore  we have $\P(W_\infty^\beta>0)\in\{0,1\}$ since the event is measurable w.r.t.\ the tail $\sigma$-algebra. We say that \textit{weak disorder} holds if $W_\infty^\beta>0$ while \textit{strong disorder} holds if $W_\infty^\beta=0$.

\smallskip Another observation from \cite{B89} is that $(W_n^\beta)_{n\in\N}$ is bounded in $L^2$ when $d\geq 3$ and $\beta$ is sufficiently small. Indeed setting
 $\beta_2\coloneqq  \sup\{ \beta\ge 0  \colon  \sup_{n\ge 0}\bbE\left[ (W_n^\beta)^2\right]<\infty\},$
an explicit computation yields that $\beta_2$ either satisfies the following identity 
\begin{equation}\label{defpbeta2}
 e^{\gl(2\beta_2)-2\gl(\beta_2)}= \frac{1}{P^{\otimes 2}(\exists n\ge 1, X^{(1)}_n= X^{(2)}_{n})},
 \end{equation}
or is equal to infinity when \eqref{defpbeta2} has no solution.
In particular, when $d\ge 3$, one has $\beta_2>0$ and  weak disorder hold for $\beta\in(0,\beta_2)$. On the other hand, $\beta_c=0$ when $d=1,2$, and it was shown in that case \cite{CH02,CSY03} that strong disorder holds for all $\beta>0$. As mentioned earlier, this work is concerned with the weak disorder phase and we thus restrict ourselves to $d\geq 3$. The following result indicates that the influence of the disorder is monotone in $\beta$.

\begin{thmx}[{\cite{CY06,JL24}}]
Assume $d\geq3$. There exists $\beta_c\in(0,\infty]$ such that weak disorder holds for $\beta<\beta_c$ and strong disorder holds for $\beta>\beta_c$. If $\beta_c<\infty$ and $\omega$ is bounded from above, i.e.\ $\supess \omega_{1,0}<\infty$, then weak disorder also holds at $\beta=\beta_c$.
\end{thmx}

\noindent It is furthermore known that $\beta_c>\beta_2$   (see \cite{BS10} for a proof of this statement when $d\ge 4$, for dimension $3$ we refer to \cite{BT10} and \cite[Theorem B]{JL24}) i.e.\ there exists an interval of $\beta$ such that weak disorder holds but $(W_n^\beta)_{n\in\N}$ is not bounded in $L^2$.
The integrability of $W^{\beta}_\infty$ in the weak disorder phase, and in particular in the interval $[\beta_2,\beta_c]$, has been an object of interest. In order to quantify it, the following integrability exponent has been introduced in \cite{J22}
\begin{equation}\label{defpbeta}
 \p(\beta)\coloneqq  \inf\left\{ p>0 \ \colon  \ \sup_{n\ge 0} \bbE\left[ (W^{\beta}_n)^p\right]=\infty \right\}.
\end{equation}
We always have $\p(\beta)\ge 1$ and as a consequence of \cite[Lemma 3.3]{CY06}, $\beta\mapsto p^*(\beta)$ is nonincreasing.
Clearly, strong disorder implies $\p(\beta)=1$ and we have $\p(\beta)\geq 2$ for $\beta<\beta_2$, but a priori there is no guarantee that $\p(\beta)>1$ in the remainder of the weak disorder. The following result summarizes what is known about the integrability of $W_\infty^\beta$.
\begin{thmx}[{\cite{CY06,J21_1}}]
Weak disorder is equivalent to $\E[\sup_n W_n^\beta]<\infty$, and in particular $W_n^\beta$ converges to $W_\infty^\beta$ in $L^1$ in weak disorder. If furthermore $\supess\omega_{1,0}<\infty$, then weak disorder implies that $\p(\beta)>1$.
\end{thmx}

\subsection{Power tail asymptotics for the partition function} \label{sec:powtail}

Our first result establishes that, up to a multiplicative constant, the tail distribution $P(W^{\beta}_\infty>u)$ decays like a power of $u$. The corresponding exponent  depends on $\beta$ and is equal to $p^*(\beta)$.
We obtain asymptotics of the same order for the suprema of partition functions 
\begin{equation}
  W^{\beta,\ast}_\infty=  \sup_{n\ge 0}   W^{\beta}_n \quad  \text{ and } \quad    \hat W^{\beta,\ast}_\infty=  \sup_{(n,x)\in \bbN \times \bbZ^d} \hat W^{\beta}_n(x),
\end{equation}
where
\begin{equation}
 \hat W^{\beta}_n(x)\coloneqq   E\left[  e^{\beta H_n(\omega,X)-n\lambda(\beta)}\ind_{\{X_n=x\}}\right].
\end{equation}
Given two positive functions $f$ and $g$ defined on $\bbR_+$ we say that $f$ and $g$ are comparable at infinity and write $f(u) \asymp g(u)$ 
if there exist $C>1$ and $u_0>0$ such that
$$\forall u\ge u_0, \quad \frac{g(u)}{C} \le f(u)\le C g(u).$$

\begin{theorem}\label{thm:tail}
 When weak disorder holds, we have the following:  
 \begin{equation}\label{asympsymp}
  \bbP\left( W^{\beta,\ast}_\infty >u\right) \asymp\bbP\left( \hat W^{\beta,\ast}_\infty >u \right) \asymp \bbP\left( W^{\beta}_\infty >u \right) \asymp  u^{-p^*(\beta)}.
 \end{equation}

\end{theorem}

\begin{rem}
Save for a couple special cases, the value of $p^*(\beta)$ is not known in general. A first  special case is the $L^2$-threshold $\beta_2$ defined in Equation \eqref{defpbeta2}: we have $p^*(\beta_2)=2$.
A second one is the critical threshold $\beta_c$. In \cite[Corollary 2.2]{JL24} it is stated that under the assumption that $\go$ is upper bounded
we have $p^*(\beta_c)=1+\frac{2}{d}$. Hence in that case Theorem~\ref{thm:tail} implies that
\begin{equation}\label{criticaldecay}
 \bbP\left( W_{\infty}^{\beta_c}\ge u\right)\asymp u^{- \frac{d+2}{d}}.
\end{equation}
 We believe that such a precise information should be helpful in any attempt to address Questions $(A)$ and $(B)$ raised in the introduction.
\end{rem}

\subsection{Overshoot considerations}\label{sec:consid}

We present next a technical result which not only plays a key role in our proof but has many other potential applications.

In recent years, many significant results concerning directed polymers have been obtained under some restriction concerning the distribution on $\go$, such as assuming that $\go$ is bounded from above \cite{J21_1,J22,J23,JL24} or some regularity on the tail distribution of $\go$ \cite{FJ23}.
One of the reasons (and in many occurrences the main one) for imposing these restrictions is to have a control on the value of the partition function as it overshoots a given threshold $A>1$. To illustrate better what we mean,
let us define $\tau_A\coloneqq \inf\{ n\ge 1 \ \colon  \ W^\beta_n\ge A\} $ (with the convention $\inf \emptyset=\infty$).
If one assumes that $\go$ is bounded from above then we have
\begin{equation}\label{tauala}
 W_{\tau_A}\le LA     \quad \text{ with } \quad  L\coloneqq  e^{\beta \mathrm{ess\, sup}(\go_{1,0})-\gl(\beta)}.
\end{equation}
The information that  $W_{\tau_A}$ is comparable to $A$ turns out to have many practical applications.
Of course \eqref{tauala} is false if the environment is not bounded from above.
 Using only the assumption \eqref{expomoment}, we obtain a result in the same spirit with a control of the $L^p$ norm of $ W_{\tau_A}$, for arbitrary $p\in[1,\infty)$, instead of the supremum norm.

\begin{proposition}\label{overshoot}
 There exists a constant $C>0$ (depending on the distribution of $\go$ and on $\beta$) such that for any $p\ge 1$ and $A>1$ we have 
\begin{equation}\label{overshooteq}
 \bbE\left[ (W^{\beta}_{\tau_A})^p \ \middle| \ \tau_A<\infty \right]< \left(2^p+C e^{\frac{\gl(2\beta p)}{2}-p\gl(\beta)}\right) A^p.
\end{equation}
\end{proposition}
\begin{rem}
 The exact expression for the constant in front of $A^p$ in \eqref{overshooteq} is not important (nor optimal), but we make a choice which is continuous and monotone in $p$ for practical purposes.
\end{rem}

\begin{rem}
The assumption  \eqref{expomoment} is not fully required by our proof and we believe that \eqref{overshooteq} should continue to hold whenever $\gl(\beta p+\delta)<\infty$ for some $\delta>0$ (similarly the validity of Theorem~\ref{thm:tail} could be extended along this line). The condition $\gl(\beta p)<\infty$ is necessary to have  $\bbE\left[ (W^{\beta}_{n})^p\right]<\infty$ and integrability slightly beyond that point is required to avoid pathological behavior. However our proof does not allow to prove such a statement since it requires moments of order $2q+2$ for $e^{\beta \omega}$ (see below Equation~\eqref{twoterms}) for a value of $q$ which itself depends on $\beta$ and on the distribution of $\go$ (cf. \eqref{choiceofq}).
\end{rem}

\noindent With the help of Proposition~\ref{overshoot}, we prove the required lower bound for $\bbP( W^{\beta,\ast}_\infty \ge u )$ in \eqref{asympsymp}.

\begin{cor}\label{overcor}
There exists a constant $c>0$ (depending on $\beta$ and on $\bbP$) such that for every $A\ge 1$ we have
\begin{equation}\label{lbd}
  \bbP\left(   \tau_A<\infty \right)=\bbP\left( W^{\beta,\ast}_\infty \ge A \right)\ge c A^{-p^*(\beta)}.
  \end{equation}
\end{cor}
\noindent Note that we do not require that weak disorder holds and that \eqref{lbd} is also valid in the strong disorder regime, in which case $p^*(\beta)=1$.
\begin{rem}
The bound \eqref{lbd} has been proved under the assumption that $p^*(\beta)\le 2$ and that the environment is upper bounded as \cite[Theorem B]{J22}.
\end{rem}

\subsection{First consequences of the main results}\label{cors}

\noindent The main results have a couple of rather direct consequences. First, we obtain a lower bound for $\p(\beta)$ in the weak disorder phase.

\begin{cor}\label{pcor}
Assume that weak disorder holds at $\beta$.
\begin{enumerate}
 \item[(i)] It holds that $\p(\beta)>1$.
 \item[(ii)] It holds that $\p(\beta)\geq 1+\frac 2d$.
\end{enumerate}

\end{cor}

Of course, claim $(ii)$ implies claim $(i)$ and there is no need to state it separately.   We included it here because of the simplicity of its proof, which is immediate from \eqref{asympsymp}.

\begin{rem}
By \cite[Corollary~2.2]{JL24}, the bound from part (ii) is sharp in the case of a bounded environment  in the sense that $\lim_{\beta \uparrow \beta_c}\p(\beta)=\p(\beta_c)=1+\frac 2 d$.  The  left-continuous behavior of $p^*(\beta)$ at criticality is obtained by combining \cite[Theorem 1.2(i)]{J23}, which states   that $p^*$ is right-continuous at every $\beta$ where $p^*(\beta)\in(1+\frac 2d,2]$, and the fact that weak disorder holds at $\beta_c$ \cite[Theorem~2.1 (ii)]{JL24}. Since \cite[Theorem 1.2(i)]{J23} assumes only exponential moments of all orders for $\go$ and since we expect  weak disorder to hold at $\beta_c$ under the same assumption (see Remark~\ref{upcoming}), we  believe that the inequality $\p(\beta)\geq 1+\frac 2d$ remains sharp when the boundedness assumption is removed.
\end{rem}

\begin{rem}
Item $(i)$ has been proved as \cite[Theorem 1.1(ii)]{J21_1} under the assumption that the environment is upper bounded and in \cite[Theorem 1.1]{FJ23} under a weaker assumption on the regularity of the tail of $\omega$ at infinity. Item $(ii)$ was proved assuming an upper bounded environment as \cite[Corollary 1.3]{J22}.
The proof of this lower bound was accomplished indirectly via the study of a fluctuation exponent associated with the partition function. This approach extends to general environment, see the comment after Corollary~\ref{corshe}, but we present an alternative, direct proof, which partially relies on ideas developed in \cite{JL24}.
\end{rem}
\noindent Next, we introduce the critical threshold for exponential growth of moments,
\begin{equation}\label{defqstar}
\q(\beta)\coloneqq \inf\left\{p\ge 0 \colon \lim_{n\to\infty}\frac 1n\log \E[(W_n^\beta)^p]>0\right\}.
\end{equation}
The existence of the above limit  follows from the subadditivity/superadditivity of the sequence $\log \E[(W_n^\beta)^p]$ (depending on whether $p\le 1$ or $p\ge 1$). The inequality $\p(\beta)\leq\q(\beta)$ is trivial.

\begin{cor}\label{tailcor}
The following hold:
\begin{enumerate}
\item[(i)] If weak disorder holds at $\beta$ then $\lim_{\beta'\uparrow \beta}\p(\beta')=\p(\beta)$. In particular, $\beta\mapsto\p(\beta)$ is left-continuous in $(0,\beta_c)$.
\item[(ii)] If weak disorder holds at $\beta$ then $\p(\beta)=\q(\beta)$.
\end{enumerate}
\end{cor}

\begin{rem}\label{upcoming}
The identity $\p(\beta)=\q(\beta)$  should also hold in the strong disorder regime (recall that in that case $p^*(\beta)=1$). Using the exponential concentration of $\log W^{\beta}_n$ around its mean \cite[Section 6]{LW09} it is possible to show that $\q(\beta)=1$ whenever $ \mathfrak f(\beta)=\lim_{n\to\infty}\frac 1n\log W_n^\beta<0$, a regime known as very strong disorder.  Moreover, the equivalence of strong disorder and very strong disorder has been established \cite[Theorem 2.1]{JL24} under the assumption that the environment is upper bounded.  This technical restriction should be removed in an upcoming work.
\end{rem}

\begin{rem}
 Item $(ii)$ has been proved under the same assumption as \cite[Theorem 1.5]{J22}. The validity of part $(ii)$ for general environments has been mentioned as an open problem in \cite[Question 6]{Zy24}.
 \end{rem}

Part $(ii)$ of the previous corollary shows that $\E[(W_n^\beta)^{p}]$ grows exponentially fast for any $p>\p(\beta)$, and by definition 
$\E[(W_n^\beta)^{p}]$ is bounded when $p<p^*(\beta)$.
When $p=\p(\beta)$, Theorem~\ref{thm:tail} immediately implies that $\lim_{n\to \infty}\E[(W_n^\beta)^{\p(\beta)}]=\infty$
but identifying the growth of $\E[(W_n^\beta)^{\p(\beta)}]$ turns out to be a challenging task. As a consequence of Theorem~\ref{thm:tail}, we can prove that this growth is at most linear in $n$.

\begin{cor}\label{corgrowth}
When weak disorder holds, then $\lim_{n\to\infty}\bbE\left[(W_n^\beta)^{\p(\beta)}\right]=\infty$. Moreover, there exists $C=C(\beta)>0$ such that, for all $n\in\N$,
\begin{equation}\label{sublin}
 \bbE\left[ (W^\beta_n)^{p^*(\beta)} \right]\le  C n.
\end{equation}

\end{cor}

 \subsection{Conjectured connection with the critical behavior of the free energy}

We believe that our main result Theorem~\ref{thm:tail} could allow us to have a more precise control on  $\bbE\left[ (W^\beta_n)^{p^*(\beta)} \right]$ than the one proved in Corollary~\ref{corgrowth}. Let us expose below a conjecture on this matter and explain the potential link with the questions $(A)$ and $(B)$ in the introduction.
We predict that  the following holds.
\begin{conjecture}\label{kappaconj}
When weak disorder holds, $\bbE\left[ (W^\beta_n)^{p^*(\beta)} \right]= n^{\kappa(\beta)+o(1)}$ as $n\to\infty$, for
\begin{align}\label{kappa_def}
\kappa(\beta):=\min\left(1, \frac{d}{2} (p^*(\beta)-1)-1\right)\in [0,1].
\end{align}
\end{conjecture}
  To motivate this claim, let us explicitly compute $\bbE\left[ (W^\beta_n)^{p^*(\beta)} \right]$ in the specific case of $\beta=\beta_2$ (recall~\eqref{defpbeta2}).
 The second moment of the partition function $\bbE[(W_n^\beta)^{2}]$ coincides with the partition function of a homogeneous pinning model 
 and thus can be computed explicitly:
$\bbE[(W_n^\beta)^{2}]$ grows exponentially in $n$ for $\beta>\beta_2$ and when $\beta=\beta_2$ we have
\begin{align}\label{critsecondmoment}
\bbE[(W_n^{\beta_2})^{2}]\sim \begin{cases} C_3 n^{1/2}&\text{ if }d=3,\\
                                    C_4\frac{n}{\log n}&\text{ if }d=4,\\
                                   C_d n&\text{ if }d\geq 5.
                                   \end{cases}
\end{align}
For a proof of \eqref{critsecondmoment} we refer to \cite[Theorem 2.2]{G07}, applied to the renewal function
$$K(n)\coloneqq  P^{\otimes 2}\left( n=\inf\{i\ge 1 \ \colon  X^{(1)}_i=X^{(2)}_i \}\right).$$
The above information, combined with the fact that weak disorder holds at $\beta_2$ and
Corollary~\ref{tailcor}, implies that $p^*(\beta_2)=q^*(\beta_2)=2$ . The asymptotic \eqref{critsecondmoment} illustrates in particular that \eqref{sublin} can be sharp in some situations (namely $\beta=\beta_2$ and $d\ge 5$) and confirms Conjecture~\ref{kappaconj} in that specific case. When $p^*(\beta)\ne 2$, we   hypothesize that
$\bbE[(W_n^{\beta})^{ \p(\beta)}]$  is comparable to the partition function of a homogeneous pinning model associated with a transient renewal process  whose interarrival-law $K^{( \p)}(n)$ is
proportional to $\sum_{x\in \bbZ} P(X_n=x)^{\p(\beta)}\asymp n^{\frac{d}{2}(1- \p(\beta))}$. The exponent $\kappa(\beta)$ corresponds to that governing the behavior of the partition function of this homogeneous pinning model at criticality (cf. \cite[Theorem 2.2]{G07}). Note in particular,  that $\kappa(\beta)=1$ when  $\frac d2(\p(\beta)-1)>1$, which corresponds to the regime where the intearrival-law $K^{( \p)}(n)$ has finite mean.

\medskip

Pushing this analogy further, we can also link $\kappa(\beta)$ with the critical behavior of Lyapunov exponents associated with $W^{\beta}_n$. We set
\begin{equation}
\tf_p(\beta)\coloneqq \lim_{n\to\infty} \frac{1}{n} \log \bbE\left[ \left(W^{\beta}_n\right)^p \right].
\end{equation}
If $\beta$ is such that $p=p^*(\beta)>1+\frac{2}{d}$, then one should have
\begin{equation}
 \tf_{p}(\beta+u)= u^{\frac{1}{\kappa(\beta)}+o(1)}
\end{equation}
as $u\downarrow 0$ (the exponent  $\frac{1}{\kappa(\beta)}$ corresponds to the critical exponent governing the free energy of the above mentioned homogeneous pinning model). The conjecture also extends to $\beta=\beta_c$.  Since we have $\kappa(\beta_c)=0$ in that case,  $\tf_{1+\frac{2}{d}}(\beta_c+u)$ should decay faster than any power of $u$.

\medskip

\noindent
Finally let us recall the definition of the free energy 
$$\mathfrak f(\beta)\coloneqq \lim_{n\to \infty}\frac 1 n \log W^{\beta}_n.$$
It has been proved \cite{JL24} (under the technical assumptions that $\go$ is upper bounded) that $\mathfrak f(\beta)<0$ if and only if $\beta>\beta_c$. The free energy curve around $\beta_c$ should be at least as smooth as that of $\tf_{1+\frac{2}{d}}$, leading to the following prediction.
\begin{conjecture}
We have
\begin{equation}
 \lim_{u\to 0+}\frac{\log |\mathfrak f(\beta_c+u)|}{\log u}=\infty.
\end{equation}

\end{conjecture}

\subsection{The fluctuation field}\label{sec:fluct}
To introduce the next corollary, let us digress a bit on the connection between the DPRE and the Stochastic Heat Equation with multiplicative noise (SHE). This connection can be seen for instance by writing the recursion equation which is satisfied by the point-to-point partition function (or some variant  of this recursion, for instance, considering shifted partition function, see Equation \eqref{eq:discr_she} below). The resulting  equation is a discrete analogue of the SHE.

\medskip

In dimension $1$, the solution of the SHE can be obtained as a scaling limit of the point-to-point partition function of the directed polymer by considering diffusive scaling and taking $\beta$ proportional to $n^{-1/4}$ where $n$ is the length of the polymer (see \cite{AKQ14}).

\medskip

When $d=2$, the SHE is  ill defined and the directed polymer model has been used as instrument to define a two dimensional version of the SHE via scaling limit (in that case $\beta$ has to be proportional to $(\log n)^{-1/2}$, see \cite{CSZ23} as well as references therein).

\medskip

In dimension $d\ge 3$, when weak disorder holds, the system homogenizes and, at first order, the disorder disappears under diffusive scaling (the scaling limit is simply the heat equation without noise). In that case, a natural question to investigate  is that of  the amplitude and distribution of the random fluctuations around this deterministic limit as done in \cite{MSZ16, GRZ18, CN21, CCN22, LZ22, J22}.

\medskip

\noindent To state the result, let us introduce the translation operator $\theta_{m,y}$. For $m\ge 0$ and $y\in \bbZ^d$ the environment $\theta_{m,y}\ \go$ is defined by setting
\begin{equation}\label{shiftoperator}
(\theta_{m,y}\ \go)_{n,x} \coloneqq  \go_{n+m,x+y}.
\end{equation}
We also let $\theta_{m,y}$ act on functions of $\go$ by setting  $\theta_{m,x} f (\go)\coloneqq  f(\theta_{m,y}\ \go).$ We are interested in investigating the scaling limit of the field $(Y_{\beta}(n,x))_{x\in \bbZ^d}$ defined by
\begin{equation}\label{ynbeta}
Y_{\beta}(n,x)\coloneqq \theta_{0,x} W_n^\beta.
\end{equation}
Note that, by time reversal, for a fixed $n$, $Y_{\beta}(\cdot,x)$ has the same distribution as $\tilde Y_{\beta}(n,\cdot)$ defined by
\begin{equation}\label{eq:discr_she}\begin{split}
   \tilde Y_{\beta}(0,x)&=1,  \quad  \quad \quad \quad \quad  \quad \quad \quad \quad  \quad \quad \quad    \forall x\in \bbZ^d,\\              
   \tilde Y_{\beta}(n+1,x)&= e^{\beta\go_{n+1,x}-\gl(\beta)} D \tilde Y_\beta(n,x).             
                \end{split}
\end{equation}
where the operator $D$ is the transition matrix of the simple random walk on $\bbZ^d$, that is to say
\begin{equation}\label{defD}
D f(x)=\frac{1}{2d} \sum_{y\in \bbZ^d} f(y)\ind_{|x-y|_1=1}.
\end{equation}
Thus \eqref{eq:discr_she} corresponds to a discrete analogue of the stochastic heat equation with multiplicative noise.
In the weak disorder phase, it has been established (see \cite[Theorem C(i)]{J22} and \cite[Theorem 2.1]{CCN22} for a continuum analogue) that homogenization occurs, i.e.\ for any $f\in C_c(\bbR^d)$ (continuous and compactly supported functions $f\colon\bbR^d\to \bbR$) the following convergence holds in~$L^1$,
\begin{equation}\label{weakconv}
\lim_{n\to \infty} n^{-d/2}\sum_{x\in\bbZ^d}f(x/\sqrt n) Y_{\beta}(n,x)=\int_{\bbR^d} f(x)\dd x.
\end{equation}
We define the fluctuation field $\mathcal X_n^\beta$ around this limit by setting for $f\in C_c(\bbR^d)$
\begin{align*}
\mathcal X_n^\beta(f)\coloneqq n^{-d/2} \sum_{x\in\bbZ^d}f(x/\sqrt n)(Y_{\beta}(n,x)-1).
\end{align*}
In the case where $\beta<\beta_2$ (that is to say $p^*(\beta)>2$), the exact scaling limit of $
\mathcal X_n^\beta$ has been identified.
More precisely it has been proved that the process $(n^{\frac{d-2}{4}}\mathcal X^{\beta}_n(f))_{f\in C_c(\bbR^d)}$ converges to 
 a random Gaussian field whose covariance can be expressed in terms of the $d$-dimensional heat kernel (see \cite[Theorem 1.1]{LZ22}).
A consequence of this is that when $p^*(\beta)>2$  and $f\not\equiv 0$, we have
\begin{equation}\label{magnitud}
 \lim_{\gep \to 0}\lim_{n\to \infty}\bbP\Big(  \gep n^{-\frac{d-2}{4}}  \leq |\mathcal X_n^\beta(f)|\le \frac{1}{\gep} n^{-\frac{d-2}{4}}\Big)=1.
\end{equation}
When $p^*(\beta)< 2$, since the field $Y_\beta(n,x)$ is not $L^2$-integrable, one may expect larger fluctuations for the field.
Indeed, in \cite[Theorem 1.1]{J22}, the correct fluctuation exponent has been identified in this regime, showing that
$|\mathcal X_n^\beta(f)|$ is of order $n^{-\xi(\beta)+o(1)}$ where 
 $$ \xi(\beta)\coloneqq  \frac{d}{2}-\frac{2+d}{2(p^*(\beta)\wedge 2)}.$$
 However, for technical reasons, the result was proved under the assumption that the environment $\go$ is upper bounded.
 The only reason for this technical limitation is that the identity $p^*(\beta)=q^*(\beta)$ was proved in \cite{J22} under this assumption.
 Hence, using Corollary~\ref{tailcor}, we can extend the validity of the result. We record this as our last corollary.

\begin{cor}\label{corshe}
When weak disorder holds we have, for any $\gep>0$ and $f\not\equiv 0$,
\begin{align}\label{fluctinprob}
\lim_{n\to\infty}\bbP\Big(n^{-\xi(\beta)-\gep}\leq |\mathcal X_n^\beta(f)|\leq n^{-\xi(\beta)+\eps}\Big)=1.
\end{align}
\end{cor}
Note that \eqref{fluctinprob} only identifies the correct fluctuation exponent but it leaves open the question of identifying the exact order of magnitude of $\mathcal X_n^\beta(f)$ (as in \eqref{magnitud}) and that of identifying the scaling limit of $\mathcal X_n^\beta(f)$ as in \cite{LZ22}. 

\begin{rem}
The combination of \eqref{weakconv} and \eqref{fluctinprob} implies that we have necessarily $\xi(\beta)\ge0$. 
This yields an alternative proof of Corollary~\ref{pcor}$(ii)$.
\end{rem}

\begin{proof}[Proof of Corollary~\ref{corshe}]
When $p^*(\beta)>2$ there is nothing to prove since \eqref{fluctinprob} follows directly from \eqref{magnitud}. 
When $p^*(\beta)\le 2$, we can use
 \cite[Theorem 1.4]{J22} which states -- under the assumption that $p^*(\beta)>1$, which holds by Corollary~\ref{pcor}(i) -- that
 \begin{equation}\begin{split}
 \lim_{n\to \infty}\bbP\Big( |\mathcal X_n^\beta(f)|\leq n^{-\xi^*(\beta)+\eps}\Big)&=1,\\
  \lim_{n\to \infty}\bbP\Big( |\mathcal X_n^\beta(f)|\geq n^{-\xi(\beta)-\eps}\Big)&=1,
  \end{split}\end{equation}
where $\xi^*(\beta)\coloneqq  \frac{d}{2}-\frac{d+2}{2q^*(\beta)}$ (recall the definition \eqref{defqstar}).
Since
by Corollary~\ref{tailcor}(ii) we have $\q(\beta)=\p(\beta)$ in weak disorder, this allows to conclude.
\end{proof}

\subsection{General reference walks}

While we have presented our results in the context of the nearest-neighbor random walk on $\bbZ^d$ -- which is the most extensively  studied setup for directed polymers -- it is worth noting that this assumption is used only marginally in the proof. For this reason all of the results presented above remain valid when \eqref{SRWdef} is replaced by
\begin{equation}\label{noasumption}
 P(X_1=x)=\nu(x)  \quad \text{ where } \nu \text{ is an arbitrary probability measure on } \bbZ^d.
\end{equation}
For the sake of completeness and to illustrate the flexibility of our proofs we register the result here.

\begin{proposition}\label{gen}
 Theorem~\ref{thm:tail}, Proposition~\ref{overshoot} and Corollaries~\ref{overcor},~\ref{tailcor} and~\ref{corgrowth} remain valid when $(X_k)_{k\ge 0}$ is only assumed to be a random walk with i.i.d.\@ increments under $P$.
\end{proposition}

\noindent The case of Corollary~\ref{pcor} is slightly different since item $(ii)$ here relies on the central limit theorem for the simple random walk.
Let us define $\eta\in [0,\infty]$ by 
\begin{equation}\label{defeta}
 \liminf_{R\to \infty} -\frac{\log \nu( \{ x : |x|>R\})}{\log R}=: \eta
\end{equation}
(with the convention that $\log 0=-\infty$).
Minor modifications in the proof allow us to obtain the following result.
\begin{cor}\label{pcor2}
If $(X_k)_{k\ge 0}$ is only assumed to be a random walk with i.i.d.\@ increments, then weak disorder implies the following 
\begin{itemize}
 \item [(i)] $\p(\beta)>1$,
 \item [(ii)] $\p(\beta)\ge 1+\frac{\min(\eta,2)}{d}.$
\end{itemize}

\end{cor}

 \begin{rem}
        We believe that, assuming sufficient regularity for the tail behavior of $\nu$, the above bound should be sharp when $\eta>0$ in the sense that one should have $\lim_{\beta\uparrow \beta_c}\p(\beta)= \p(\beta_c)=1+\frac{\min(\eta,2)}{d}$ .
        However, there are significant technical obstacles to proving such a statement, including the generalization of \cite[Theorem B]{JL24} in the case when $\frac{\min(\eta,2)}{d}>\frac{2}{3}$. We leave this question for future work.
       \end{rem}

Finally, we observe that as a consequence of Corollary~\ref{pcor2}, we can show that there exists a family of random walks for which $\beta_c=\beta_2$.
\begin{cor}\label{pcor3}
In the case when $\frac{\min(\eta,2)}{d}=1$ we have $\beta_c=\beta_2$.
\end{cor}
There are many cases where  $\beta_c=\beta_2=0$. This corresponds to the case where the difference between two independent copies of $X$ ($X^{(1)}-X^{(2)}$ using the notation in Equation \eqref{defpbeta2}) is recurrent and includes already known results such as the  Cauchy random walk in dimension $1$ \cite[Proposition 1.13]{W16} and the simple random walk in dimension $2$ \cite[Theorem~2.3]{CSY03} as well as any random walk with $\eta>2$ for $d=2$  (in this case also, a proof that $\beta_c=0$ can be obtained with only minor modifications to the argument presented in \cite{CSY03}). More interestingly, this also includes cases where $\beta_2>0$. For instance on may consider the probabilies defined on $\bbZ$ and $\bbZ^d$ by
\begin{equation}\begin{split}
 \nu_1(x)&= c_1 \frac{(\log (2+|x|))^2}{|1+x|^2} \quad \text{ with } \quad  x\in \bbZ,\\
  \nu_2(x)&= c_2 \frac{\log (2+|x|)}{|1+x|^4} \quad \text{ with } \quad  x\in \bbZ^d,\\
\end{split}\end{equation}
where $c_1$ and $c_2$ are normalizing constants that make $\nu_2$ and $\nu_2$ probabilities. To our knowledge, this is the first instance where the (finite) critical temperature for a model of directed polymer in an i.i.d.\ random environment is computed explicitly.

\subsection{Organization of the paper}
As explained in the above introduction, the most novel mathematical contributions are Theorem~\ref{thm:tail} and Proposition~\ref{overshoot} and the major part of the paper is devoted to their proofs.

\medskip

Most of the corollaries have been previously proved under technical restriction.
The proofs for many of these corollaries presented below simply replicate the original argument (found in \cite{J21_1}, \cite{J22} or \cite{FJ23}), incorporating in it the new input given by either Theorem~\ref{thm:tail} or Proposition~\ref{overshoot}. An exception to this  is the proof of Corollary~\ref{pcor}(ii) (and Corollary~\ref{pcor2}(ii)) which greatly differ from the one found in \cite{J22}. Instead of an indirect proof based on the fluctuation exponent, we present a more direct one which relies on ideas developed in \cite{JL24}.

\medskip

Let us finally explain how the remaining sections are organized.
In Section~\ref{sec:overshoot}, we prove Proposition~\ref{overshoot} and  Corollary~\ref{overcor}.
In Section~\ref{sec:pfthm}, we prove Theorem~\ref{thm:tail} and Corollaries~\ref{overcor} and~\ref{pcor}$(i)$. The remaining corollaries are proved in Section~\ref{sec:cors}.

\medskip

Proposition~\ref{gen} and Corollary~\ref{pcor2}(i) do not require separate proofs since the corresponding proofs presented in Sections~\ref{sec:overshoot},~\ref{sec:pfthm} and~\ref{sec:cors} do not rely on the fact that $(X_k)_{k\ge 0}$ is the simple random walk. The only points that require a (minor) adaptation are addressed in Remarks~\ref{minoradapt}, \ref{nochange} and~\ref{otherrem}.
The more substantial changes required to prove Corollaries~\ref{pcor2}(ii) and \ref{pcor3} are detailed in Sections~\ref{proofpcor2} and \ref{proofpcor3}.

\subsection{Notation}
Throughout the paper, we make use of the notation $\lint a,b\rint\coloneqq  [a,b]\cap\bbZ$.
We let $\mu^{\beta}_n$ denote the endpoint distribution associated with the polymer measure $P^{\beta,\go}_{n}$, that is to say
\begin{equation}\label{npoint}
 \mu^{\beta}_n(x)=P^{\beta,\go}_{n}(X_n=x).
\end{equation}

\section{Proof of the results of Section~\ref{sec:consid}}\label{sec:overshoot}

\subsection{Probability of doubling the mean of a linear combination of i.i.d.\ random variables}

We start this section introducing a key technical lemma used in the proof.
Let us introduce first some context. Recalling the definition~\eqref{defD} and \eqref{npoint}, we have
\begin{equation}\label{stoop}
\frac{W^{\beta}_n}{W^\beta_{n-1}}= \sum_{x\in \bbZ^d} D \mu^{\beta}_{n-1}(x) e^{\beta \go_{x,n}-\gl(\beta)}.
\end{equation}
The coefficients $D \mu^{\beta}_{n-1}(x)$ are $(\cF_{n-1})$-measurable while the variables  $(e^{\beta \go_{x,n}-\gl(\beta)})_{x\in \bbZ^d}$ are i.i.d.\,and independent of $\cF_{n-1}$ -- they also have  mean $1$ and moments of all orders. The following result can be used to estimate the conditional probability of  multiplying the partition function at step $n$  by a factor larger than $2$.

\begin{lemma}\label{technik}
Let $(Y_i)_{i\ge 1}$ be a sequence of i.i.d.\ nonnegative random variables with moments of all orders and such that $\bbE[Y_1]=1$. Then for any $q\ge 1$ there exists $C_q$ (depending on the distribution of $Y_1$) such that  for any sequence $(\alpha_i)_{i\ge 1}$ of nonnegative  real numbers
  satisfying $\sum_{i\ge 1} \alpha_i\le 1$ 
 \begin{equation}
  \bbP\left(  \sum_{i\ge 1} \alpha_i Y_i \ge 2 \right)\le C_q \alpha_{\max}^q,
 \end{equation}
where $\alpha_{\max}=\max_{i\ge 1} \alpha_i$.

\end{lemma}

\begin{proof}
We want to derive the inequality from a Chernov type bound. Since a priori the variables $Y_i$ are not exponentially integrable, we first need to apply a truncation. We observe that
\begin{equation}\label{twoterms}
\bbP\left(  \sum_{i\ge 1} \alpha_iY_i \ge 2  \right)
\le \bbP\left( \max_{i\ge 1} \alpha_i Y_i > \sqrt{\alpha_{\max}}\right)+ \bbP\left(  \sum_{i\ge 1} (\alpha_i Y_i)\wedge \sqrt{\alpha_{\max}} \ge 2  \right).
\end{equation}
We can bound the first term in the r.h.s.\ as follows
\begin{equation*}
\bbP\left( \max_{i\ge 1} \alpha_i Y_i > \sqrt{\alpha_{\max}}\right)\le \sum_{i\ge 1} \bbP\left( \alpha_i Y_i >  \sqrt{\alpha_{\max}}\right)
 \le \bbE[Y^{2q+2}_1] \sum_{i\ge 1} \frac{\alpha_i^{2q+2}}{\alpha^{q+1}_{\max}} \le \bbE[Y^{2q+2}_1] \alpha^{q}_{\max},
\end{equation*}
where the last inequality simply uses that $\alpha_i^{2q+2}\le \alpha^{2q+1}_{\max} \alpha_i.$
Now to bound the second term in the r.h.s.\ of \eqref{twoterms}, we observe that for any $\gl>0$ we have
\begin{equation}
 \bbP\left(  \sum_{i\ge 1} (\alpha_iY_i)\wedge \sqrt{\alpha_{\max}} \ge 2 \right)\le \bbE\left[ e^{\gl \left[  \sum_{i\ge 1} ((\alpha_i Y_i)\wedge \sqrt{\alpha_{\max}}\;)-2\right]}\right].
\end{equation}
Using the inequality $e^u\le 1+u+u^2$ valid for $u\le 1$ we obtain that for all $\gl\le \alpha_{\max}^{-1/2}$, 
\begin{equation}
 \bbE\left[e^{\gl  ((\alpha_i Y_i)\wedge \sqrt{\alpha_{\max}}\;)}\right]\le \bbE\left[1+ \gl\alpha_i Y_i+ \gl^2\alpha^2_i Y_i^2\right]\le e^{\gl\alpha_i+\gl^2\alpha_i^2 \bbE[Y_1^2]}.
\end{equation}
This implies that 
\begin{equation}
 \bbE\left[ e^{\gl \left[  \sum_{i\ge 1} (\alpha_i Y_i)\wedge \sqrt{\alpha_{\max}}-2\right]}\right]\le e^{\gl\left[\sum_{i\ge 1}\alpha_i-2  \right]+\gl^2\bbE[Y^2_1]\sum_{i\ge 1}(\alpha_i)^2}.
\end{equation}
Using $\sum_{i\ge 1}\alpha_i\le 1$ and $\sum_{i\ge 1} \alpha^2_i\le \alpha_{\max}$ and taking $\gl=\alpha_{\max}^{-1/2}$, and assuming, without loss of generality that $2q\in \bbN$, we obtain
\begin{equation}
	\bbP\left(  \sum_{i\ge 1} (\alpha_iY_i)\wedge \sqrt{\alpha_{\max}} \ge 2 \right)\le e^{\E[Y_1^2] -\alpha_{\max}^{-1/2}} \leq e^{\bbE[Y_1^2]}  (2q)!\alpha_{\max}^{q},
	\end{equation}
	 where in the second inequality we simply used the fact that $e^{\sqrt{x}}\ge \frac{x^{{q}}}{2q!}$  for $x\geq 0$.
This concludes the proof.
\end{proof}
 
\subsection{Proof of Proposition~\ref{overshoot}}

Since the result is trivial if $\omega$ is bounded from above, we may assume that $\P(\omega_{1,0}>t)>0$ for all $t>0$.
	We set $\tau\coloneqq \tau_A$ for notational simplicity.
 We discard first the contribution of ``small overshot'' above the value $A$ using the following decomposition 
\begin{equation}\label{splitting}\begin{split}
  \bbE\left[ (W^{\beta}_{\tau})^p \ind_{\{\tau<\infty\}} \right]& =\bbE\left[ (W^{\beta}_{\tau})^p \left(\ind_{\{\tau<\infty \ ; \ W_\tau\in[A,2A)\}} +\ind_{\{\tau<\infty \ ; \ W_\tau \ge 2A\}} \right)\right]\\
&\le (2A)^p\bbP(\tau<\infty)+ \bbE\left[ (W^{\beta}_{\tau})^p\ind_{\{\tau<\infty \ ; \ W_\tau \ge 2A\}} \right].
  \end{split}\end{equation}
 With the above the proof of \eqref{overshooteq} is reduced to showing that
  \begin{equation}\label{tipee}
   \bbE\left[ (W^{\beta}_{\tau})^p\ind_{\{\tau<\infty \ ; \ W_\tau \ge 2A\}} \right]\le C e^{\frac{\gl(2p\beta)}{2}-p\gl(\beta)}\bbP[\tau<\infty].
  \end{equation}
We decompose according to the value of $\tau$ and obtain
\begin{equation}\begin{split}\label{tac}
 \bbE\left[ (W^{\beta}_{\tau})^p\ind_{\{\tau<\infty \ ; \ W_\tau \ge 2A\}} \right]
 &=\sum_{n\ge 1} \bbE\left[ (W^{\beta}_{n})^p\ind_{\{\tau\ge n \ ; \ W^{\beta}_n \ge 2A\}} \right]\\
  &=\sum_{n\ge 1} \bbE\left[ \bbE\left[ (W^{\beta}_{n})^{p}\ind_{\{W^{\beta}_n \ge 2A\}} \ \middle| \cF_{n-1} \right] \ind_{\{\tau\ge n\}} \right].
\end{split}\end{equation}
Next we apply Cauchy-Schwarz 
\begin{equation}\label{tic}
 \bbE\left[ (W^{\beta}_{n})^p\ind_{\{ W^{\beta}_n \ge 2A\}}  \ \middle| \ \cF_{n-1} \right]\le  \sqrt{ \bbE\left[ (W^{\beta}_{n})^{2p}  \ \middle| \ \cF_{n-1} \right]\bbP\left( W^{\beta}_n\ge 2A  \ \middle| \ \cF_{n-1} \right) }.
\end{equation}
Now setting $\xi_{x}\coloneqq e^{\beta \go_{n,x}-\gl(\beta)}$ and recalling \eqref{stoop} we obtain that
(using  Jensen's inequality for the probability $D \mu^{\beta}_{n-1}(x)$) on the event $\tau\ge n$
\begin{equation}\begin{split}\label{toc}
  \bbE\left[ (W^{\beta}_{n})^{2p} \ \middle| \ \cF_{n-1} \right]&=  (W^{\beta}_{n-1})^{2p}  \bbE\Bigg[ \bigg(\sum_{x\in \bbZ^d}  D \mu^{\beta}_{n-1}(x) \xi_x\bigg)^{2p} \ \bigg| \ \cF_{n-1} \Bigg] \\ &\le  A^{2p}e^{\gl(2p\beta)-2p\gl(\beta)} .
\end{split}\end{equation}
Combining \eqref{tac}, \eqref{tic} and \eqref{toc}, the proof of \eqref{tipee} (and thus that of the proposition) is now reduced to showing
\begin{equation}\label{tipee2}
 \sum_{n\ge 1}\bbE\left[\left(\sqrt{\bbP\left( W^{\beta}_n\ge 2A  \ \middle| \ \cF_{n-1} \right)}\right) \ind_{\{\tau\ge n\}}\right]\le C \bbP[ \tau< \infty].
\end{equation}
To estimate the conditional probability above we rely on Lemma~\ref{technik} with coefficients $\alpha_x\coloneqq A^{-1} D\hat W^{\beta}_{n-1}(x)$, $Y_x=\xi_x$ and
\begin{equation}\label{choiceofq}
q\coloneqq   -4\log_2 \kappa \quad   \text{ with }  \quad \kappa\coloneqq \bbP( e^{\beta \go_{0,1}-\gl(\beta)}\ge 4d )\wedge \frac 12,
\end{equation}
($\kappa>0$ since we assumed that $\go$ was unbounded from above).
We obtain there exists  $C_1>0$ (depending on $\beta$ and the distribution of $\go$) such that, on the event 
$\tau\ge n$ (which guarantees that $W^{\beta}_{n-1}\le A$ and hence that $\sum \alpha_x\le 1$) we have
\begin{equation}
\bbP\left( W^{\beta}_n\ge 2A  \ \middle| \ \cF_{n-1} \right)\le C_1 \left( \max_{x\in \bbZ^d}  A^{-1} D\hat W^{\beta}_{n-1}(x) \right)^q.
\end{equation}
Thus in view of \eqref{tipee2} we can conclude if we show that for some $C>0$ we have
\begin{equation}\label{tipee3}
  \sum_{n\ge 1}\bbE\left[\left(   A^{-1} \max_{x\in \bbZ^d} D\hat W^{\beta}_{n-1}(x) \right)^{q/2} \ind_{\{\tau\ge n\}}\right]\le C \bbP[ \tau< \infty].
\end{equation}
We are going to prove this bound with $\max_{x\in \bbZ^d} D\hat W^{\beta}_{n-1}(x)$ replaced by the larger quantity $\max_{x\in \bbZ^d} \hat W^{\beta}_{n-1}(x)$.
Recalling the definition of $\kappa$ in \eqref{choiceofq}, we observe that for any $k\ge 1$,
\begin{equation}\label{goodestimate}
 \bbP\left[ W^{\beta}_{n+k-1}\ge 2^{k}   \max_{x\in \bbZ^d} \hat W^{\beta}_{n-1}(x)   \ \middle| \ \cF_{n-1}\right]\ge \kappa^k.
\end{equation}
Indeed, assuming that $\max_{x\in \bbZ^d} \hat W^{\beta}_{n-1}(x)$ is attained at $x_0$, the event on the right-hand side is satisfied if
\begin{equation}\label{unit}
\forall i \in \lint 1,k\rint, \quad  e^{\beta \go_{n+i-1,x_0+(i-1){\bf e}_1}-\gl(\beta)}\ge 4d,
\end{equation}
where ${\bf e_1}=(1,0,\dots,0)$ denotes the first coordinate unit vector. Of course \eqref{goodestimate} remains valid if $k$ is replaced by any $(\cF_{n-1})$-measurable random quantity.
With this in mind we set 
$$\Theta_n\coloneqq   \left \lceil \log_2 \left( \frac{A}{\max_{x\in \bbZ^d}  \hat W^{\beta}_{n-1}(x)}\right)\right\rceil \ge 1.$$
We obtain from \eqref{goodestimate} that on the event $\{\tau\ge n\}$ we have
\begin{equation}
\kappa^{-\Theta_n} \bbE\left[ \ind_{\lint n, n+\Theta_n-1 \rint}(\tau) \ \middle| \ \cF_{n-1} \right]\ge  \kappa^{-\Theta_n}  \bbP\left( W^{\beta}_{n+\Theta_n-1}\ge A   \ \middle| \ \cF_{n-1}\right) \ge  1.
\end{equation}
Hence,  we have (recall from \eqref{choiceofq} that $2^{q/2}\kappa=\kappa^{-1}$)
\begin{equation}\label{dasplit}\begin{split}
 \bbE\left[ \left(  A^{-1}\max_{x\in \bbZ^d}   \hat W^{\beta}_{n-1}(x) \right)^{q/2} \ind_{\{\tau\ge n\}}
 \right]&\le  \bbE\left[  \ind_{\{\tau\ge n\}}  \left(2^{1-\Theta_n}\right)^{\frac{q}{2}} \kappa^{-\Theta_n} \bbE\left[ \ind_{\lint n, n+\Theta_n-1 \rint}(\tau) \ \middle| \ \cF_n \right] \right] \\
 &= 2^{q/2} \bbE\left[  \ind_{\lint n, n+\Theta_n-1\rint}(\tau)  \kappa^{\Theta_n}\right]\\
  &\le 2^{q/2} \bbE\left[   \kappa^{1+\tau-n} \ind_{\{n\leq \tau <\infty\}}\right].
\end{split}
 \end{equation}
Since $\kappa\le 1/2$, we have
$\sum_{n\geq 1}    \kappa^{1+\tau-n} \ind_{\{ n\leq \tau <\infty\}}\le \ind_{\{\tau<\infty\}},$ so that summing \eqref{dasplit} over $n$  yields
\begin{equation}
\sum_{n\geq 1}\bbE\left[ \left(  A^{-1}\max_{x\in \bbZ^d}   \hat W^{\beta}_{n-1}(x) \right)^{q/2} \ind_{\{\tau\ge n\}}
 \right]\le 2^{q/2}\bbP\left(  \tau <\infty\right),
\end{equation}
concluding our proof.
\qed

\begin{rem}\label{minoradapt}
To obtain the generalization mentioned in Proposition~\ref{gen}, the quantity $4d$ appearing in \eqref{choiceofq} and \eqref{unit} has to be replaced by   $2(\max_{x\in\bbZ^d}\nu(\{x\}))^{-1}$. Additionally  ${\bf e}_1$ in \eqref{unit} has to be replaced by $\argmax_{x\in\bbZ^d} \nu(\{x\})$.
\end{rem}

\subsection{Proof of Corollary~\ref{overcor}}
Recalling \eqref{shiftoperator} and \eqref{npoint} we have for any $k\le n$
\begin{equation}
 W^{\beta}_n/W^{\beta}_k\coloneqq  \sum_{x\in \bbZ^d}  \mu^{\beta}_k(x)\theta_{k,x}(W^{\beta}_{n-k}).
\end{equation}
Hence for $p\ge 1$, Jensen's inequality applied twice yields
\begin{equation}\label{jenso}
 \E\left[\left(W^{\beta}_n/W^{\beta}_k\right)^p \ \middle| \ \cF_{k} \right]\le \E[(W^{\beta}_{n-k})^p]
 \le \E[(W^{\beta}_{n})^p].
\end{equation}
Hence for any $A>1$, $p> 1$ and $n\in\N$, we have
\begin{equation}\label{ineqq}
\begin{split}
	\E[(W^{\beta}_n)^p]&= \bbE\left[(W^{\beta}_n)^p \ind_{\{\tau_A>n\}}  \right]  +\sum_{k=1}^n\E[(W^{\beta}_n)^p\1_{\{\tau_A=k\}}]\\
		 &\leq A^p+\sum_{k\leq n}\E\left[(W^{\beta}_{k})^p\ind_{\{\tau_A=k\}}\E\left[\left(W^{\beta}_n/W^{\beta}_k\right)^p \ \middle| \ \cF_{k} \right]
        \right]\\
		 &\leq A^p+\E[(W^{\beta}_{n})^p] \sum_{k\leq n}\E[(W^\beta_k)^p\1_{\{\tau_A=k\}}]\\
		  &\leq A^p+\E[(W^{\beta}_{n})^p] \bbE[(W^\beta_{\tau_A})^p\ind_{\{\tau_A<\infty\}} ] \\
		 &\leq A^p+C_pA^p\P(\tau_A<\infty)\E[(W^{\beta}_n)^p],
\end{split}
\end{equation}
where we have used \eqref{jenso} in the first inequality and \eqref{overshooteq} in the last one.

Then, since $\bbE[(W^{\beta}_n)^p]$ is unbounded for $p> p^*(\beta)$, \eqref{ineqq} implies that
 \begin{equation}
  C_pA^p\P(\tau_A<\infty)\ge 1.
 \end{equation}
 Taking the limit as $p\downarrow p^*(\beta)$ yields
the desired result with $c= (C_{p^*(\beta)})^{-1}$.
\qed

\subsection{Proof of Corollary~\ref{pcor}(i)}
We simply observe that
$$\E\left[\sup_{n\ge 0  } W^{\beta}_n\right]= 1+\int^{\infty}_1 \bbP\left[ \tau_A<\infty\right]\dd A.$$
Since by \cite[Theorem 1.1(i)]{J21_1},  $\E[\sup_{n\ge 0} W^{\beta}_n]<\infty$ in the weak disorder phase, in view of \eqref{lbd} we must necessarily have $p^*(\beta)>1$.

\begin{rem} \label{nochange}
  The proof of Corollary~\ref{pcor2}(i) is identical. Since \cite[Theorem~1.1(i)]{J21_1}  is stated only for nearest-neighbor random walk one may apply instead directly \cite[Theorem~2.1(i)]{J21_1} to the martingale $(W^{\beta}_n)$.
 \end{rem}

\qed

\section{Proving Theorem~\ref{thm:tail}}\label{sec:pfthm}

 \subsection{Organizing the proof of Theorem~\ref{thm:tail}}

 To prove \eqref{asympsymp} we are going to prove a total of $6$ inequalities (some which are redundant, but this allows for a smoother presentation).
Before we begin the formal proof, let us briefly expose which inequalities we prove.
First we observe that 
  \begin{equation}\label{trivial}
 \hat W^{\beta,\ast}_\infty\le W^{\beta,\ast}_{\infty}  \quad  \text{ and } \quad W^{\beta}_{\infty}\le W^{\beta,\ast}_{\infty},
 \end{equation}
 and hence two inequalities are immediate.
A third inequality, that is \eqref{lbd} in Corollary~\ref{overcor}, has been proved in the previous section.
 This leaves us with three remaining inequalities to prove.
 The first one is proved in Section~\ref{sec:smult}, using the supermultiplicativity of the point-to-point partition functions.

\begin{lemma}\label{supermult}
In the weak disorder regime,  we have for any $u>1$,
\begin{equation}\label{upbd}
\bbP\left( \hat W^{\beta,\ast}_\infty>  u \right)\le u^{-p^*(\beta)}.
\end{equation}
\end{lemma}

\noindent The next lemma does not exactly match an inequality in \eqref{asympsymp} (because of the factor $4$ in the r.h.s.) but combined with the others it is sufficient to prove the theorem. It is proved in Section~\ref{sec:wwstar}.

\begin{lemma}\label{lem:wwstar}
 There exists a constant $c$ such that for every $u> 1$
 \begin{equation}
  \bbP\left[  W^{\beta}_{\infty}\ge u\right] \ge  c\bbP[ W^{\beta,\ast}_{\infty}\ge 4u ].
\end{equation}
\end{lemma}

The most technical part of the proof is the comparison of the tail of $\hat W^{\beta,\ast}_\infty$ (the maximum over point-to-point partition functions) with that of $W^{\beta,\ast}_{\infty}$ (the maximum over point-to-plane partition functions). We achieve this by proving a localization result for the endpoint at time $\tau_u$, which is interesting in its own right. The outcome of our proof (presented in Sections~\ref{sec:compariso} and~\ref{sec:part2}) is the following comparison which combined with the five previous inequalities allows to complete the proof of Theorem~\ref{thm:tail}.

 \begin{proposition}\label{prop:compariso}
There exist constants $C>1$, $\delta>0$ and a closed set $\mathcal U\subset[1,\infty)$ satisfying
\begin{equation}\label{maxspace}
 \forall v\ge 1, \quad [v,Cv]\cap \mathcal U\ne \emptyset,
\end{equation}
(the set $\{ \log u \ \colon  \ u\in \mathcal U\}$ has no gap larger than $\log C$) such that the following holds
 \begin{equation}\label{localize}
\forall u \in \mathcal U, \quad    \bbP\left[ \max_{x\in \bbZ^d}  \mu_{\tau_{u}}(x)\ge \delta  \ \middle| \  \tau_u<\infty \right] \ge \delta.
   \end{equation}

 \end{proposition}
\begin{rem}
  Observe that the above proposition is an intermediate result on the path of obtaining Theorem~\ref{thm:tail}. Once Theorem~\ref{thm:tail} is proven, one obtains that for $\delta>0$ sufficiently small
  \begin{equation}\label{evenbetter}
 \forall u\ge 1, \quad  \bbP\left[ \max_{x\in \bbZ^d}  \mu_{\tau_{u}}(x)\ge \delta  \ \middle| \  \tau_u<\infty \right] \ge \delta,
\end{equation}
with no need for a restriction to a specific set $\mathcal U$ (cf. Remark~\ref{deuxieme}).
 \end{rem}

\noindent We end this subsection by checking that, indeed, the combination of all elements exposed above yields our main result.
\begin{proof}[Proof of Theorem~\ref{thm:tail}]

Let us first show that $\bbP\left[ W^{\beta,\ast}_{\infty}\ge  u  \right]$ and $\bbP\left[ \hat W^{\beta,\ast}_{\infty}\ge  u  \right]$ are comparable to $u^{-p^*}$. As two of the four required inequalities are provided by \eqref{upbd} and \eqref{lbd}, it only  remains  to show that for some constant $\kappa>0$ we have
\begin{equation}\label{lasttwo}\begin{split}
 \forall u\ge 1, \quad \quad \bbP\left[ W^{\beta,\ast}_{\infty}\ge  u  \right]\le  \frac{1}{\kappa}u^{-p^*(\beta)},  \\
   \forall u\ge 1, \ \ \ \quad  \bbP\left[\hat  W^{\beta,\ast}_{\infty}\ge  u  \right]\ge  \kappa u^{-p^*(\beta)}.
\end{split}\end{equation}
From Proposition~\ref{prop:compariso}, we can find $C,\delta$ and $\mathcal U$ such that \eqref{maxspace} and \eqref{localize} hold.
Let us define $u'$ and $u''$  to be points in $\mathcal U$  which approximate $u$ well
\begin{equation}
 u'\coloneqq  \max\{ v\in \cU \ \colon  \ v\le u\} \quad \text{ and } \quad   u''\coloneqq  \min\{ v\in \cU \ \colon  \ \delta v\ge u\}.
\end{equation}
 In particular $u'\geq u/C$ and $u''\leq Cu/\delta$ due to \eqref{maxspace}.
Both $\bbP\left[ W^{\beta,\ast}_{\infty}\ge  u  \right]$ and $\bbP\left[ \hat W^{\beta,\ast}_{\infty}\ge  u  \right]$ are decreasing in $u$.
Thus, at the cost of changing the value of the constant $\kappa$,  it is sufficient to show that the first line in \eqref{lasttwo} is valid for $u'$ and that the second line is valid for $\delta u''$.
From \eqref{localize},  and since  $\hat W^{\beta}_{\tau_v}(x)=  W^{\beta}_{\tau_v} \mu_{\tau_v}(x) \ge v\mu_{\tau_v}(x)$,   we have for any $v\in \cU$
\begin{equation}\label{preineq}
 \bbP\left( \hat W^{\beta,\ast}_{\infty}\ge \delta v  \right)\ge \bbP\left( \max_{x\in \bbZ^d} \hat W^{\beta}_{\tau_v}(x)\ge \delta v\  ; \ \tau_v<\infty  \right) \ge 
 \delta \bbP\left(\tau_v<\infty \right).
\end{equation}
Hence using \eqref{lbd} (in the third inequality) and writing $p^*$ for $p^*(\beta)$ for better readability,  we obtain that
\begin{equation}
\bbP\left( \hat W^{\beta,\ast}_{\infty}\ge u  \right)\ge \bbP\left( \hat W^{\beta,\ast}_{\infty}\ge \delta u''  \right)
\ge \delta \bbP(\tau_{u''}<\infty) \geq  \delta c(u'')^{-\p}\geq cC^{-\p}\delta^{1+\p} u^{-\p}.
\end{equation}
On the other hand using \eqref{preineq} (in the second inequality below) and \eqref{upbd}  (third inequality) we obtain,
\begin{equation}
 \bbP\left[ W^{\beta,\ast}_{\infty}\ge  u  \right]\le  \bbP\left(\tau_{u'}<\infty \right)\le \delta^{-1}\bbP\left( \hat W^{\beta,\ast}_{\infty}\ge \delta u'  \right)\le \delta^{-1-p^*} (u')^{-p^*}\le  C^{p^*}\delta^{-1-p^*}  u^{-p^*},
\end{equation}
finishing the proof of \eqref{lasttwo}.

\smallskip To conclude, let us prove that  $\bbP\left[ W^{\beta}_{\infty}\ge  u  \right]\asymp u^{-p^*}$. The upper bound is a direct consequence of the first line in \eqref{lasttwo} since $W^{\beta}_{\infty}\le W^{\beta,\ast}_{\infty}$. The lower bound is obtained by combining Lemma~\ref{lem:wwstar} with \eqref{lbd}.
\end{proof}

\subsection{Proof of Lemma~\ref{supermult}}\label{sec:smult}

We can focus on the case $p^*(\beta)>1$ since when $p^*(\beta)=1$ the result is an immediate consequence of  Doob's martingale inequality.
 The result follows from the following claim. Setting
  $\zeta(u)\coloneqq  \bbP\left( \sup_{n\ge 0, x\in\bbZ^d} \hat W^{\beta}_n(x)> u \right)$, we have
\begin{equation}\label{supermulti}
\forall u, v > 1, \quad \zeta(uv)\ge  \zeta(u)\zeta(v).
\end{equation}
Indeed, for any $u>1$, $k\ge 1$ and $p\in (1,p^*(\beta))$ we have
\begin{equation}
 \left(\zeta(u) u^p\right)^k\le \zeta(u^k) u^{kp} \le  \bbE\left[  \left(  \hat W^{\beta,\ast}_\infty\right)^p \right]  \le \bbE\left[  \left(  W^{\beta,\ast}_\infty\right)^p \right]<\infty.
\end{equation}
By taking $k\to\infty$, this implies that $\zeta(u)\le u^{-p}$, and we obtain the result by taking the limit $p\uparrow p^*(\beta)$.
Now let us justify \eqref{supermulti}. Set $B_u\coloneqq   \{ \sup_{n\ge 0, x\in\bbZ^d} \hat W^{\beta}_n(x)> u\}$, so that $\zeta(u)=\bbP(B_u)$. We define
$$(\sigma,Z)\coloneqq \min \left\{ (n,x)\in \N\times\bbZ^d \ \colon  \   \hat W^{\beta}_n(x)>u  \right \}, $$
where $\min$ refers to the minimal element of the set for  the lexicographical order on $\N\times\bbZ^d$. By convention  $(\sigma,Z)=(\infty,0)$ on the event $B^{\complement}_u$. Recalling the definition \eqref{shiftoperator} we have, for any $n,m\geq 0$ and $x,z\in\bbZ^d$,
\begin{equation}
 \hat W^{\beta}_{n+m}(x+z)\ge \bbE\left[ e^{\beta H_{n+m}(X)-(n+m)\gl(\beta)}\ind_{\{X_m=z, X_{n+m}=x+z\}}\right]= \hat W^{\beta}_{m}(z) \theta_{m,z}  \hat W^{\beta}_{n}(x).
 \end{equation}
Hence  $\hat W_{n+m}(x+z)> uv$ if $\hat W_{m}(z)> u$ and  $\theta_{m,z}\hat W_{n}(x)> v$.
Decomposing over the possible values of $(\sigma,Z)$ we  have
\begin{equation}
\uniontwo{m\ge 1, z\in \bbZ^d}{n\ge 1, x\in \bbZ^d} \{(\sigma,Z)=(m,z) \}\cap \theta_{m,z} B_v\subset  B_{uv}.
\end{equation}
Now the events $\{(\sigma,Z)=(m,z) \}$ and  $\theta_{m,z} B_v$ rely on disjoint regions of the i.i.d.\ environment and are hence independent.
Combining this observation with translation invariance and the fact that  $\{(\sigma,Z)=(m,z)\}_{m\ge 1, z\in \bbZ^d}$ is a partition of the event $B_u$, we obtain that
\begin{align}
 \bbP(B_{uv})&\ge \sum_{m\ge 1, z\in \bbZ^d}\bbP\left( \{(\sigma,Z)=(m,z)\} \cap\theta_{m,z}B_v\right)\notag  \\
 &= \sum_{m\ge 1, z\in \bbZ^d}\bbP\left( (\sigma,Z)=(m,z)\right) \bbP(B_v) = \bbP(B_u)\bbP(B_v).\tag*{\qed}
\end{align}

 \subsection{Proof of Lemma~\ref{lem:wwstar}}\label{sec:wwstar}

Set $A=4u$. On the event  $\tau_{A}<\infty$, recalling the definition~\eqref{defD} and \eqref{npoint}, we have
\begin{equation}
 W^{\beta}_{\infty}= W^{\beta}_{\tau_A}\sum_{x\in \bbZ^d} \mu_{\tau_A}(x) \theta_{\tau_A,x} W^{\beta}_\infty
 \ge A \sum_{x\in \bbZ^d} \mu_{\tau_A}(x) \theta_{\tau_A,x}W^{\beta}_\infty.
\end{equation}
Let us set $Z^A(x)\coloneqq  \left(\theta_{\tau_A,x} W^{\beta}_\infty\right)$. Note that $(Z^A_x)_{x\in \bbZ}$ is independent of $\cF_{\tau_A}$ and distributed like $(Z(x))_{x\in \bbZ^d}\coloneqq \left(\theta_{0,x} W^{\beta}_\infty\right)_{x\in \bbZ^d}$.
For this reason we have for any $u\ge 0$
\begin{equation}
\bbP\left(  W^{\beta}_{\infty}\ge u\right) \ge \bbP[\tau_A<\infty ] \inf_{ \alpha\in \cP(\bbZ^d)} \bbP\left(  \sum_{x\in \bbZ^d}\alpha(x)Z(x)\ge 1/4\right)
\end{equation}
where the infimum is taken over all probability measures $\alpha$ on $\bbZ^d$. To conclude we just need to show that this infimum is positive.
Now if one sets $\mathcal W_{\alpha}\coloneqq \sum_{x\in \bbZ^d}\alpha(x)Z(x)$,
it is immediate to check that the collection of variables  $(\mathcal W_{\alpha})_{\alpha \in \cP(\bbZ^d)}$ is uniformly integrable.
Indeed if we let $\varphi:\bbR_+\to \bbR_+$ be a convex function such that  $\E[\varphi(W_\infty^\beta)]<\infty$ and $\lim_{x\to \infty}\varphi(x)/x=\infty$ (such a $\varphi$ exists since $W_\infty^\beta\in L^1$. For example, we may choose $\varphi(x)\coloneqq \sum_{k}(x-n_k)_+$, where the $n_k$ are such that $\lim_{k\to\infty}n_k=\infty$ and $\E[(W_\infty^\beta-n_k)_+]\leq k^{-2}$), Jensen's inequality yields that for any  $\alpha\in \cP(\bbZ^d)$
\begin{equation}
 \bbE\left[  \varphi( \mathcal W_{\alpha})\right]\le \sum_{x\in \bbZ^d} \alpha(x) \bbE\left[ \varphi(Z(x))\right] =\E[\varphi(W_\infty^\beta)],
\end{equation}
implying the desired uniform integrability.
Since $\bbE[ \mathcal W_{\alpha}]=1$, using uniform integrability there exists $M>0$ such that for every $\alpha\in \cP(\bbZ^d)$ 
\begin{equation}\begin{split}
\bbE[ \mathcal W_{\alpha} \ind_{\{\mathcal W_{\alpha}\in [1/4, M]\}}]&=1-
\bbE[ \mathcal W_{\alpha} \ind_{\{\mathcal W_{\alpha}< 1/4\}}]-\bbE[ \mathcal W_{\alpha} \ind_{\{\mathcal W_{\alpha}> M\}}]\\
&\ge 3/4 - \bbE[ \mathcal W_{\alpha} \ind_{\{\mathcal W_{\alpha}> M\}}]\ge 1/2,
\end{split}
\end{equation}
 and thus we have $$\bbP(\mathcal W_{\alpha} \ge 1/4)\ge  \frac{1}{M}\bbE[ \mathcal W_{\alpha} \ind_{\{\mathcal W_{\alpha}\in [1/4, M]\}}]\geq  \frac{1}{2M}.$$  \qed

 \subsection{Proof of Proposition~\ref{prop:compariso}}\label{sec:compariso}

We start by fixing two parameters
\begin{equation}\label{defD2}
q>p^*(\beta) \quad  \text{ and } \quad D\coloneqq 3C_{2q} 2^{2q}
\end{equation}
 where $C_{2q}\ge 1$ comes from Proposition~\ref{overshoot}.
 The reason for this peculiar choice for $D$ will become apparent in the course´ of the proof of Proposition~\ref{u1isgood} below.
We start with the geometric sequence/set $\cU_0\coloneqq \{  D^k\colon k\ge 0 \}$. Note that $\cU_0$ clearly satisfies the density requirement \eqref{maxspace}. We are going to show \eqref{localize} is valid for a ``large'' subset of $\cU_0$ such that \eqref{maxspace} remains true.
 Given $q>\p$, we set
\begin{equation}
 \cV\coloneqq  \{ u>1\ : \ \bbP(\tau_{Du}<\infty)\ge  D^{-q}\bbP(\tau_u<\infty)\} \quad \text{ and } \quad  \cU_1\coloneqq \cU_0\cap \cV.
\end{equation}

Note that $\cU_1$ is infinite. Indeed if this were not the case, then there would exist some $k_0$ such that the sequence
 $(D^{kq} \bbP(\tau_{D^k}<\infty))_{k\ge k_0}$ is non-increasing (and hence bounded). Since $q>\p$, this would be a contradiction to \eqref{lbd}.
The next step is then to prove that $u\in \cV$ satisfies the desired localization property \eqref{localize}.
This result is the most technically demanding and its proof is postponed to the next subsection.
\begin{proposition}\label{u1isgood}
There exist $\delta>0$  such that for every $u\in \cV$
\begin{equation}\label{smalldelta}
 \bbP\left( \max_{x\in \bbZ^d} \mu_{\tau_u}(x)\ge \delta  \ \middle| \ \tau_u<\infty\right)\ge \delta.
\end{equation}
\end{proposition}

\begin{rem}\label{deuxieme}
 Note that \textit{a posteriori}, Theorem~\ref{thm:tail} implies that $\cV=[1,\infty)$ if $D$ is chosen sufficiently large.
 Hence Proposition~\ref{u1isgood} implies that \eqref{evenbetter} holds.
\end{rem}

\noindent Finally to conclude we need to show that there are no big gaps in $\cU_1$. This is the role of the following lemma.
\begin{lemma}\label{k0gap}
 There exists an integer $k_0\ge 1$ such that  
 \begin{equation}\label{sequenz}
 \forall u \in \cU_1, \quad  \{ uD^{i} \ : \ i\in \lint 1, k_0\rint\}\cap \cU_1 \ne \emptyset
 \end{equation}
In particular $\cU_1$ satisfies \eqref{maxspace} with $C= D^{k_0}$.
\end{lemma}
\noindent Proposition~\ref{prop:compariso} follows immediately from the combination of Proposition~\ref{u1isgood} and Lemma~\ref{k0gap}.

\begin{proof}[Proof of Lemma~\ref{k0gap}]
For $k_0$ to be chosen later, let us assume that we can find $u\in \cU_1$ such that \eqref{sequenz} does not hold. Then we have
 \begin{equation}
  \bbP\left(  \tau_{uD^{k_0+1}}<\infty \right)\le D^{-q k_0}   \bbP\left(  \tau_{uD}<\infty \right)\le D^{-qk_0}   \bbP\left(  \tau_{u}<\infty \right).
 \end{equation}
 In the first inequality, we have used that $uD^i\notin \mathcal V$ for $i=1,...,k_0$, and in the last inequality we had to use a trivial bound since $u\in\mathcal V$.
 On the other hand, since $u\in \cU_1$, Proposition~\ref{u1isgood} implies that \eqref{localize} holds. Recalling the computation \eqref{preineq} this  implies that
 \begin{equation}
   \bbP\left(  \tau_{u}<\infty \right)\le \delta^{-1} \bbP\left( \hat  W^{\beta,\ast}_{\infty}\ge \delta u \right)\le \delta^{-p^*-1} u^{-p^*},
 \end{equation}
where the last bound is simply \eqref{upbd}.
We conclude that 
\begin{equation}
   \bbP\left(  \tau_{uD^{k_0+1}}<\infty \right)\le  D^{-q{( k_0+1)}}  \delta^{-p^*-1}  u^{-p^*}.
\end{equation}
Now recalling that $q>\p$, the bound contradicts \eqref{lbd} if $k_0$ is large enough that
$$ D^{-q k_0}  \delta^{-p^*-1}\le c D^{-p^* { (k_0+1)}},$$
where $c$ is the constant in \eqref{lbd}.
\end{proof}

 \subsection{Proof of Proposition~\ref{u1isgood}}\label{sec:part2}

 We are going to prove the contraposition
\begin{equation}\label{contradelta}
 \bbP\left( \max_{x\in \bbZ^d} \mu_{\tau_u}(x)\ge \delta  \ \middle| \ \tau_u<\infty\right)\le \delta \quad \Longrightarrow \quad \frac{\bbP\left( \tau_{Du}<\infty \right)}{\bbP\left(\tau_{u}<\infty \right)}< D^{-q}.
\end{equation}
Note that on the event $\tau_{Du}<\infty$, either $W_{\tau_u}$ is much larger than $u$ or 
 $(W_{n+\tau_u}/W_{\tau_u})_{n\ge 1}$ overshoots $2$,
 \begin{equation}\label{stepp1}
\frac{\bbP\left( \tau_{Du}<\infty \right)}{\bbP\left(\tau_{u}<\infty \right)}\le \bbP\left(\ W_{\tau_u}\geq \frac {Du}2 \  \middle | \  \tau_u<\infty \right)+\bbP\left( \sup_{n\ge 0} \frac{W_{n+\tau_u}}{W_{\tau_u}}\ge 2 \  \middle | \ \tau_u<\infty\right).
 \end{equation}
The first term can be controlled by Proposition~\ref{overshoot}. We have
\begin{align}\label{first}
\P\left(W_{\tau_u}\geq \frac{Du}2\ \middle | \ \tau_u<\infty\right) \le  \left(\frac{Du}2\right)^{-2q} \bbE\left[ |W_{\tau_u}|^{2q} \ | \ \tau_u<\infty\right]     \leq C_{2q}2^{2q}D^{-2q}\le \frac{D^{-q}}{3},
\end{align}
where the last inequality is valid thanks to our choice for $D$ \eqref{defD2}.
To complete the proof of \eqref{contradelta}, we are going to show that, for a value of $\delta$ to be determined,
\begin{equation}\label{contra2}
 \bbP\left( \max_{x\in \bbZ^d} \mu_{\tau_u}(x)\ge \delta  \ \middle| \ \tau_u<\infty\right)\le \delta \quad \Longrightarrow \quad \bbP\left( \sup_{n\ge 0} \frac{W_{n+\tau_u}}{W_{\tau_u}}\ge 2 \  \middle | \ \tau_u<\infty\right)\le \frac{D^{-q}}{3}.
\end{equation}
Recalling \eqref{shiftoperator}, we first note that
\begin{equation}\label{decompo}
\frac{W_{n+\tau_u}}{W_{\tau_u}}= \sum_{x\in \bbZ^d} \mu_{\tau_u}(x)\theta_{\tau_u,x} W^{\beta}_n.
\end{equation}
Now by strong Markov's property $(\theta_{\tau_u,x} W^{\beta}_n)_{n\ge 0, x\in \bbZ^d}$ is independent of $\cF_{\tau_u}$ and distributed like $(\theta_{0,x}  W^{\beta}_n )_{n\ge 0, x\in \bbZ^d}$. This allows us to consider  $\mu_{\tau_u}(x)$ as a deterministic measure.
We define  $\gG \colon \bbR_+\to [0,1]$ by
\begin{equation}
 \Gamma(\eta)\coloneqq  \sup_{\alpha \in \cP(\bbZ^d) \ : \  |\alpha|_{\infty}\le \eta} \bbP\left( \sup_{n\ge 0}\sum_{x\in \bbZ^d}\alpha(x) \theta_{0,x}  W^{\beta}_n  \ge 2\right),
\end{equation}
where $|\alpha|_{\infty}=\max_{x\in \bbZ^d}\alpha(x)$.
In particular, we have
\begin{multline}\label{stepp2}
  \bbP\left( \sup_{n\ge 0} \frac{W_{n+\tau_u}}{W_{\tau_u}}\ge 2  \ \middle | \  \tau_u<\infty\right)\\
  \leq \bbE\left[ \Gamma(|\mu_{\tau_u}|_{\infty}) \ | \ \tau_u<\infty \right]
  \leq \Gamma(\eta)+ \bbP\left(  \max_{x\in \bbZ^d} \mu_{\tau_u}(x)> \eta \ \middle| \ \tau_u<\infty  \right).
\end{multline}
 To prove the implication \eqref{contra2}, we show the following result (which states that the probability of doubling the partition function starting from a spread-out measure is small).

\begin{lemma}\label{gammasmall}
We have
$\lim_{\eta \to 0}\gG(\eta)=0$.

\end{lemma}

\noindent  Indeed, using the lemma we can find $\eta$ such that $\Gamma(\eta)\leq \frac{D^{-q}}6$ and \eqref{contra2} then follows by setting $\delta\coloneqq \eta\wedge \frac{D^{-q}}{6}$ and applying \eqref{stepp2}.
\qed

\begin{proof}[Proof of Lemma~\ref{gammasmall}]

Using Markov's inequality (note that the variable in the expectation is nonnegative  since $W_0^\beta=1$) we have
\begin{equation}
\gG(\eta)\le \sup_{\alpha \in \cP(\bbZ^d) \ : \  |\alpha|_{\infty}\le \eta}  \bbE\left[ \sup_{n\ge 0}\sum_{x\in \bbZ^d}\alpha(x) \ \theta_{0,x}  W^{\beta}_n-1\right],
\end{equation}
so that we just need to show that the r.h.s.\ goes to zero when $\eta\to 0$.
The first step is to reduce the range over which the $\sup$ is taken. For any $M$, we have
\begin{multline}\label{twotermshere}
  \bbE\left[ \sup_{n\ge 0}\sum_{x\in \bbZ^d}\alpha(x) \theta_{0,x}  W^{\beta}_n\right]\le
  \bbE\left[ \sup_{n\in \lint 0,M\rint}\sum_{x\in \bbZ^d}\alpha(x) \theta_{0,x}  W^{\beta}_n \right]\\
  +\sum_{x\in \bbZ^d}\alpha(x)\bbE\left[ \sup_{i\ge 0}\theta_{0,x}  W^{\beta}_{M+i}-\theta_{0,x} W^{\beta}_{M}\right]
\end{multline}
Now since $\alpha$ is a probability, by translation invariance we have
\begin{equation}
 \sum_{x\in \bbZ^d}\alpha(x)\bbE\left[ \sup_{i\ge 1} \theta_{0,x}  W^{\beta}_{M+i}-\theta_{0,x} W^{\beta}_{M}\right]
= \bbE\left[ \sup_{i\ge 0}   W^{\beta}_{M+i}- W^{\beta}_{M}\right].
\end{equation}
Note that we have, almost surely,
\begin{equation}
\lim_{M\to \infty}\left(\sup_{i\ge 0}   W^{\beta}_{M+i}- W^{\beta}_{M}\right)=0  \quad \text{ and } \quad 
\sup_{i\ge 0}  \left| W^{\beta}_{M+i}- W^{\beta}_{M}\right|\le W^{\beta,\ast}_\infty
\end{equation}
Since by \cite[Theorem 2.1(i)]{J21_1},  $W^{\beta,\ast}_\infty\in L^1$, we can use dominated convergence to obtain that 
for $M$ sufficiently large we have
\begin{equation}\label{blopp}
 \bbE\left[ \sup_{i\ge 0}   W^{\beta}_{M+i}- W^{\beta}_{M}\right]\le \gep/2
\end{equation}
and we are left with estimating the first term in \eqref{twotermshere}.
First observe that (by Cauchy-Schwarz)
\begin{equation}
 \bbE\left[ \sup_{n\in \lint 0,M\rint}\sum_{x\in \bbZ^d}\alpha(x) \theta_{0,x}  W^{\beta}_n -1\right]
 \le \bbE\left[ \sup_{n\in \lint 0,M\rint} \left(\sum_{x\in \bbZ^d} \alpha(x) \theta_{0,x}  W^{\beta}_n -1\right)^2\right]^{1/2}.
\end{equation}
Then we use Doob's $L^2$-inequality (for the martingale $\sum_{x\in \bbZ^d} \alpha(x) \theta_{0,x}  W^{\beta}_n -1$)
and obtain that
\begin{equation}
\bbE\left[ \sup_{n\in \lint 0,M\rint}\left( \sum_{x\in \bbZ^d} \alpha(x) \theta_{0,x}  W^{\beta}_n -1\right)^2\right]
\le 4 \bbE\left[\left(  \sum_{x\in \bbZ^d} \alpha(x) \theta_{0,x}  W^{\beta}_M -1\right)^2\right].
\end{equation}
Using the bounds $\alpha(x')\le |\alpha|_{\infty}$ and
$\bbE\left[\left( W^{\beta}_M -1\right)^2\right]\le\bbE\left[( W^{\beta}_M)^2\right]\le   e^{M\left(\gl(2\beta)-2\gl(\beta) \right)}$ we have
\begin{equation}\begin{split}
 \bbE\left[\left(  \sum_{x\in \bbZ^d} \alpha(x) \theta_{0,x}  W^{\beta}_M -1\right)^2\right]
 &\le  \sumtwo{x,x'\in \bbZ^d}{|x-x'|\le 2M} \alpha(x)\alpha(x')\bbE\left[\left( W^{\beta}_M -1\right)^2\right]\\
 &\le (4M+1)^d |\alpha|_{\infty}  e^{M\left(\gl(2\beta)-2\gl(\beta) \right)}.
\end{split}\end{equation}
 Together with \eqref{blopp}, this guarantees that
\begin{equation}
 \sup_{\alpha \in \cP(\bbZ^d) \ : \  |\alpha|_{\infty}\le \eta}  \bbE\left[ \sup_{n\ge 0}\sum_{x\in \bbZ^d}\alpha(x) \ \theta_{0,x}  W^{\beta}_n-1\right]\le \gep
\end{equation}
for $\eta=\frac{\gep^2}{4}(4M+1)^{-d} e^{-M\left(\gl(2\beta)-2\gl(\beta) \right)}$ and concludes the proof.
\end{proof}

\section{Proof of the corollaries from Section~\ref{cors}}\label{sec:cors}

\subsection{Preparation for the proof of Corollary~\ref{pcor}(ii)}

 We  prove that   $p^*(\beta)< 1+\frac{2}{d}$ implies strong disorder. To do so we borrow a couple of elements of proof from \cite{JL24}. This includes using a characterization of strong disorder via the size-biased measure, as well as an estimate of the tail distribution of the supremum of the partition function $W^{\beta}_n$ taken over finitely many values of $n$.

 \medskip

 The size-biased measure $\tilde \bbP_n$ defined by
 $\dd \tilde \bbP_n\coloneqq  W^{\beta}_n \dd \bbP$. When weak disorder holds, $\tilde \bbP_n$ converges when $n\to \infty$ to a measure $\tilde \bbP_{\infty}$ which is absolutely continuous w.r.t. $\bbP$. On the other hand, when strong disorder holds, $\tilde \bbP_n$ and $\bbP$ become increasingly singular as $n$ grows. Hence to prove that strong disorder holds, it is sufficient to find an event $A_n$ which has large probability under $\bbP$ and small probability under $\tilde \bbP_n$. The following lemma, although not strictly necessary, helps to understand this dichotomy. Its proof can be found in \cite[Lemma 3.2]{JL24}.

 \begin{lemma}\label{finitevolume}
  For any event $A$ we have
  \begin{equation}
\bbE\left[ (W^{\beta}_n)^{1/2}\right] \le \sqrt{ \bbP(A^{\complement})}+ \sqrt{\tilde \bbP_n(A)}.
  \end{equation}

 \end{lemma}
 \noindent Our goal is thus to identify a sequence of events $A_n$ such that 
 \begin{equation}\label{tozero}
 \lim_{n\to \infty}\bbP(A^{\complement}_n)=\lim_{n\to \infty} \tilde \bbP_n(A_n)=0.
 \end{equation}
Before introducing the event $A_n$, we state a key technical result that will be essential to obtain an estimate on   $\tilde \bbP_n(A_n)$.

\begin{lemma}\label{tailfinite}
Assume that weak disorder holds at $\beta$. Given $\gep>0$, there exist constants $C_{\gep}>0$ and $u_0(\gep)>0$ such that for every $u\ge u_0(\gep)$
\begin{equation}\label{finivotail}
 \exists m\in \lint 0, C_{\gep} \log u\rint, \quad  \bbP \left( W^{\beta}_m\ge u \right) \ge u^{-p^*(\beta)-\gep}.
\end{equation}

\end{lemma}

The above result is the analogue of \cite[Proposition 4.2]{JL24} which states that \eqref{finivotail} is valid when strong disorder holds
(in which case $p^*(\beta)=1$)
under the additional assumption that the environment is upper bounded. The proof of the present lemma is almost identical, except that we use Theorem~\ref{thm:tail} (more specifically, the asymptotic behavior of $\hat W^{\beta,*}_n$) in place of \cite[Theorem 4.4]{JL24} as an input. We therefore only sketch the argument.

\begin{proof}
First, by an inductive argument similar to the one used to prove \eqref{supermulti} (we refer to the proof of \cite[Proposition 6.1]{JL24} for details), we obtain, for any $k,T\in\N$ and $v>1$,
\begin{align}\label{inductive}
  \P\left(\max_{n\in\lint 1, kT \rint, x\in \bbZ^d} \hat W_n^\beta(x) > v^k\right)\geq \P\left(\max_{n\in\lint 1,T\rint,x\in\bbZ^d}\hat W^{\beta}_n(x)> v\right)^k.
\end{align}
Then we fix $v_1$ large enough that $v_1^{-\eps/2}\leq \frac c2$, where $c$ is such that
$\P\left(\hat W_\infty^{\beta,*}>u\right)\geq cu^{-\p}$
holds for all $u>1$ (such a constant exists due to Theorem~\ref{thm:tail}). By choosing $T$ large  enough, we have \begin{equation}\label{forall}
\P\left(\max_{n\in\lint 1,T\rint,x\in\bbZ^d}\hat W^{\beta}_n(x)> v_1\right)\geq \frac{c}{2} v_1^{-p^*}.
\end{equation}
Now given $u$ we set $k=k_u\coloneqq  \lceil \log u /\log v_1 \rceil$.
By a union bound and \eqref{inductive}--\eqref{forall}, we obtain
\begin{multline}\label{tyui}
\max_{n\in\lint 1,kT\rint}\P\left(W_n^\beta >u\right)\ge \frac{1}{kT}
\P\left(\max_{n\in\lint 1,kT\rint} W_n^\beta >u\right)\\ \ge \frac{1}{kT} \P\left(\max_{n\in\lint1, kT\rint,x\in \bbZ^d} \hat W_n^\beta(x) >v_1^k\right)\geq \left(\frac c2v_1^{-\p}\right)^k\geq v_1^{-k(p^*+\eps/2)} \ge \left(uv_1\right)^{-(p^*+\eps/2)}.
\end{multline}
The second inequality follows from the fact that $v^k_1\geq u$ (by definition of $k_u$), the third one is due to \eqref{forall} and the fourth one is a consequence of  $v_1^{-\eps/2}\leq \frac c2$.
One can then check that \eqref{tyui} implies the desired claim  with $u_0= (v_1)^{\frac{2p^*}{\gep}+1}$ and $C_\eps\coloneqq \frac T{\log v_1}$.
\end{proof}
The last ingredient we need is a well-known construction which allows for a simple representation of the size-biased measure, which is sometimes called the \textit{spine construction}, see for instance \cite[Lemma 1]{B04}.
We define $(\hat \go_i)_{i\ge 1}$ as a sequence of i.i.d.\@ random variables with marginal distribution given by
\begin{equation}\label{tilted}
\hat \bbP( \hat \go_1\in \cdot)= \bbE\left[ e^{\beta\go_{1,0}-\gl(\beta)}\ind_{\{\go_{1,0}\in \cdot\}}\right],
\end{equation}
 and $X$ a simple random walk ($\hat \bbP$ and $P$ denote the respective distributions). Given $\go$, $\hat \go$ and $X$, all sampled independently, we define a new environment $\tilde \go=\tilde \go(X,\go,\hat \go)$  by
\begin{equation}
 \tilde \go_{i,x}\coloneqq \begin{cases}
	 \go_{i,x}  & \text{ if } x\ne X_i,\\
	 \hat \go_i  & \text{ if } x= X_i.
            \end{cases}
\end{equation}
In words,  $\tilde \go$ is obtained by tilting the distribution of the environment on the graph of $(i,X_i)_{i=1}^\infty$.
The distribution of $\tilde \go$ under $P\otimes \bbP\otimes \hat \bbP$ corresponds to that of $\go$ under the size-biased measure.
\begin{lemma} \label{sbrepresent}It holds that
\begin{equation}
\tilde \bbP_n[\ (\go_{i,x})_{i\in \lint 1, n\rint, x\in \bbZ^d }\in \cdot \ ]=P\otimes \bbP\otimes \hat \bbP[ (\tilde \go_{i,x})_{i\in \lint 1, n\rint, x\in \bbZ^d } \in \cdot ].
\end{equation}
\end{lemma}
\noindent We refer to \cite[Lemma 3.3]{JL24} for a proof of the above using the same notation.

\subsection{Proof of Corollary~\ref{pcor}(ii)} \label{sec:pstarlow}

 We assume  that  $p^*=p^*(\beta)\in(1,1+\frac{2}{d})$ and we will obtain a contradiction by showing that strong disorder holds.
We fix three parameters
\begin{equation}\label{parameters}
     \alpha\in\left(\frac{2+ d}{2\p},\frac{1}{\p-1}\right),\quad\xi\in\left(\frac 12,\frac{\alpha\p- 1}d\right),\quad \text{ and }\quad \zeta\in\left(\p,1+\frac 1\alpha\right).
    \end{equation}
 Our assumption guarantees that the specified intervals are not empty.
Using Lemma~\ref{tailfinite}, for $n\ge n_0$ sufficiently large and some $C>0$
we let $m$ be such that 
\begin{equation}\label{suchthatthat}
 m\le  C \log n \quad \text{ and } \quad \bbP \left( W^{\beta}_m\ge n^{\alpha} \right) \ge n^{-\alpha \zeta}. 
\end{equation}
Now we define the event $A_n$ as follows
\begin{equation}
 A_n\coloneqq  \left\{ \forall (r,x)\in  \lint0, n-m\rint\times \lint -n^{\xi}, n^{\xi}\rint^d, \quad  \theta_{r,x} W^{\beta}_m< n^{\alpha}\right\}.
\end{equation}
Our aim is to prove that both $\bbP(A^{\complement}_n)$ and $\tilde \bbP_n(A_n)$ tend to zero (cf.\ \eqref{tozero}). Since $((W_n^{\beta})^{1/2})_{n\in\N}$ is uniformly integrable, $\lim_{n\to\infty}\bbE[(W_n^{\beta})^{1/2}]=0$ implies strong disorder,  see Lemma~\ref{finitevolume}, as desired.
 Using a union bound and translation invariance we have (using Theorem~\ref{thm:tail})
\begin{equation}
 \bbP(A^{\complement}_n)\le { n(2n^{\xi}+1)^d}\bbP\left(  W^{\beta}_m\ge n^{\alpha} \right) \le C n^{d\xi+1-\alpha p^*}.
\end{equation}
With our choice of parameters \eqref{parameters}, the exponent in the r.h.s.\ is negative and hence  $\bbP(A^{\complement}_n)$ tends to zero. The remainder of the proof is to show that the same holds for  $\tilde\bbP(A_n)$.

\medskip

The key idea  is to show that, in some sense, under the size biased measure $\tilde \bbP_n$, we can extract $n/m$ variables amongst the collection $(\theta_{r,x} W^{\beta}_m)_{x\in \bbZ^d, r\in \lint 0,n-m\rint}$, whose distribution is i.i.d.\@ and given by $\tilde \bbP_m\left( \tilde W_m\in \cdot \right)$.
This is where the spine representation from Lemma~\ref{sbrepresent} is useful.
Let $\theta_{r,x}  \tilde W^{\beta}_{m}$ denote the partition function built with environment $\tilde \go(X,\go,\hat \go)$ that is,
\begin{equation}
 \theta_{r,x} \tilde W^{\beta}_{m}\coloneqq  E'\left[ e^{\sum_{i=1}^m \beta\tilde \go_{m+i,x+X'_i}-m\lambda(\beta)} \right]
\end{equation}
where $X'$ is a simple random walk with the same law as $X$ ($X$ appears in the definition $\tilde \go$ so it is not available as a variable of integration).
From Lemma~\ref{sbrepresent} we have
$$\tilde \bbP_n(A_n)
=\hat \bbP\otimes \bbP \otimes P\left( \tilde \go \in A_n \right).$$
Now we observe that if $\tilde \go$ satisfies $A_n$ then either $\theta_{r,x} \tilde W^{\beta}_m$ has to be small when $(r,x)$ runs along the graph of $X$, or the graph of $X$ must leave the box $\lint1, n\rint\times \lint -n^{\xi}, n^{\xi}\rint^d$, i.e.
\begin{equation}
\{ \tilde \go \in A_n \} \subset  \left\{ \forall r\in \lint 0,n-m\rint, \  \theta_{r,X_{r}} \tilde W^{\beta}_{m}< n^{\alpha} \right\}\cup \big\{ \exists i\in \lint 1,n\rint,   X_{i}\notin  \lint -n^{\xi}, n^{\xi}\rint^d  \big\}.
\end{equation}
Setting $j_n\coloneqq \lfloor n/m \rfloor$ and restricting the range of $r$ to multiples of $m$ we obtain that
\begin{multline}\label{firstterm}
\hat \bbP\otimes \bbP \otimes P\left( \tilde \go \in A^{\complement}_n \right)\le P\left(  \exists i\in \lint 1,n\rint,  \   X_{i}\notin \lint -n^{\xi}, n^{\xi}\rint^d \right)
\\ \quad \quad + \hat \bbP\otimes \bbP \otimes P\left( \forall i\in \lint 0,j_n-1\rint, \  \theta_{im,X_{im}} \tilde W^{\beta}_{m}< n^{\alpha}  \right).
 \end{multline}
Since $\xi> 1/2$, the probability of the first term goes to zero and we can conclude by showing that
\begin{equation}
 \lim_{n\to \infty} \hat \bbP\otimes \bbP \otimes P\left( \forall i\in \lint 0,j_n-1\rint, \  \theta_{im,X_{im}} \tilde W^{\beta}_{m}< n^{\alpha}  \right)=0.
\end{equation}
By construction (using the Markov property for the random walk $X$) the sequence of environments
$$\left(\left(\theta_{im,X_{im}} \tilde \go_{n,x}\right)_{(n,x)\in \lint 1,m\rint\times \bbZ^d}\right)_{i\ge 0}$$
is independent and identically distributed.
In particular, the sequence of variables $(\theta_{im,X_{im}} \tilde W^{\beta}_{m})_{i\ge 1}$ is i.i.d.\@ and  from Lemma~\ref{sbrepresent} the marginal distribution is given by
\begin{equation}
 P\otimes \bbP\otimes \hat \bbP\left(\tilde W^{\beta}_{m}\in \cdot \right)=\tilde \bbP_m\left( W^{\beta}_m\in \cdot\right).
\end{equation}
Observing that from \eqref{suchthatthat} we have
$$\tilde \bbP_m( W_m^\beta\ge n^{\alpha}) =\bbE\left[ W_m^\beta \ind_{\{W_m^\beta\ge n^{\alpha}\}}\right]\ge n^{-\alpha(\zeta-1)}, $$
 we obtain that for any $j\ge 1$ we have
\begin{equation}\label{indux}
\hat \bbP\otimes \bbP \otimes P\left( \forall i\in \lint 0,j-1\rint, \  \theta_{im,X_{im}} \tilde W^{\beta}_{m} < n^{\alpha} \right)=
\tilde \bbP_m\left( W_m^\beta< n^{ \alpha}\right)^j \le e^{-j n^{-\alpha(\zeta-1)}}.
\end{equation}
We apply the above formula for $j_n$ and we can conclude using the fact that $\alpha(\zeta-1)<1$ (due to our choice of \eqref{parameters}), hence $\lim_{n\to \infty} j_n    n^{\alpha(\zeta-1)}=\infty$.
\qed

\subsection{Proof of Corollary~\ref{pcor2}(ii)} \label{proofpcor2}

When $\eta=0$ there is nothing to prove. When $\eta>0$ the proof is very similar to that of the previous section. We assume that $ p^*\in(1,1+\frac{2\wedge\eta}{d})$ and set
\begin{equation}\label{paramiter}
 \alpha\in \left(\frac{\eta\wedge 2+d}{(\eta\wedge 2)\p},\frac{1}{\p-1}\right),\quad \xi\in\left(\frac{1}{\eta\wedge 2},\frac{\alpha\p-1}{d}\right),\quad\zeta\in\left(\p,1+\frac1\alpha\right).
    \end{equation}
    Using  Lemma~\ref{tailfinite}, for $n\ge n_0$ sufficiently large and some $C>0$
we let $m$ be such that 
\begin{equation}\label{suchthat}
 m\le  C \log n \quad \text{ and } \quad \bbP \left( W^{\beta}_m\ge n^{\alpha} \right) \ge n^{-\alpha \zeta}.
\end{equation}
When $\eta\le 1$, we  define the event $A_n$ as 
\begin{equation}
 A_n\coloneqq  \left\{ \forall (r,x)\in  \lint0, n-m\rint\times\lint -n^{\xi}, n^{\xi}\rint^d, \quad  \theta_{r,x} W^{\beta}_m< n^{{\alpha}}\right\}.
\end{equation}
When $\eta>1$ we set  
\begin{equation}
 A_n\coloneqq  \left\{ \forall (r,x)\in  \mathcal C, \quad  \theta_{r,x} W^{\beta}_m< n^{{ \alpha}}\right\},
\end{equation}
where
$$\mathcal C\coloneqq  \left\{ (r,x) \colon \ r\in \lint0, n-m\rint \quad \text{ and } \quad  |x-r\bbE[X_1]|\le n^{\xi}\quad   \right\}. $$
The proof that $\bbP(A_n^{\complement})$ and $\tilde \bbP(A_n)$  converge to zero is then identical to that presented in the previous section.
     The only technical point is to check that $\xi> 1/(\eta\wedge 2)$ implies that
    \begin{equation}
     \lim_{n\to \infty}  P\left(  \exists i\in \lint 1,n\rint,  \   |X_{i}-\bbE[X_i]\ind_{\{\eta>1\}}|>n^{\xi} \right)=0.
    \end{equation}
        This is a rather classical computation but let us include it here for completeness. Without loss of generality (using a union bound on the coordinates) we can assume that $d=1$.
    If $\eta>2$  the above is a consequence of Donsker's Theorem for a random walk with finite second moment, so we may also assume that $\eta\le 2$.
We consider $\xi' \in (1/\eta,\xi)$ and   the truncated random walk $ X'_i\coloneqq \sum_{j=1}^n (X_{i}-X_{i-1})\ind_{\{|X_{i}-X_{i-1}|\le n^{\xi'}\}} $ and observe that
\begin{multline}
    P\left(  \exists i\in \lint 1,n\rint,  \   |X_{i}-\bbE[X_i]\ind_{\{\eta>1\}}|>  n^{\xi} \right)\\
    \le    P\left(  \exists i\in \lint 1,n\rint,  \   |X_{i}-X_{i-1}|> n^{\xi'} \right) +P\left(  \exists i\in \lint 1,n\rint,  \    |X'_{i}-\bbE[X_i]\ind_{\{\eta>1\}}|> n^{\xi} \right).
    \end{multline}
    The first term tends to zero since
\begin{equation}
  P\left(  \exists i\in \lint 1,n\rint,  \   |X_{i}-X_{i-1}|> n^{\xi'} \right)\le n P(|X_1|> n^{ \xi'}).
\end{equation}
As for the second term, when $\eta\le 1$ we simply observe that for $n$ sufficiently large
$$\bbE\left[\max_{i\in \lint 0,n\rint}|X_i'|\right]\le n\bbE[|X_1'|]\le n^{1+\xi'(1-\eta)+o(1)}=o(n^{\xi})$$
and apply Markov's inequality. When $\eta\in (1,2]$ we observe that $$|\bbE[X_i]-\bbE[X_i']|\le n^{1-\xi'(\eta-1)+o(1)}  =o(n^{\xi})$$ and that 
$\mathrm{Var}(X_n')\le n^{\xi'(2-\eta) +1+o(1)} =o(n^{2\xi})$ and conclude using Doob's inequality for the martingale $(X_i'-\bbE[X_i'])_{i=1}^n$. \qed

\subsection{Proof of Corollary \ref{pcor3}}\label{proofpcor3}
We are going to show that if $\beta>\beta_2$ then $\p(\beta)<2$.
Then, combining the  assumption $(\eta\wedge 2)/d=1$ with Corollary \ref{pcor2}, we obtain that necessarily we have strong disorder at $\beta$, and since this applies for all $\beta>\beta_2$ the desired result follows.
The important part of the proof is to show that if  $\beta>\beta_2$ then there exist $m\ge 1$  and $x\in \bbZ^2$ such that
\begin{equation}\label{tropp}
 \bbE[(\hat W^\beta_{m})(x)^2]\ge 3.
\end{equation}
Indeed, arguing by continuity similarly to the proof of Corollary~\ref{tailcor}, if \eqref{tropp} holds one can find  $p<2$ such that  $\bbE[(\hat W^\beta_{n})(x)^p]\ge 2$ and hence, restricting the partition function to the event $\{X_{jm}=jx, \forall j\in \lint 1, k\rint\}$, we obtain that
\begin{equation}
\bbE[(W^\beta_{km})^p]\ge \bbE\left[(\hat W^\beta_{m})(x)^p \right]^k \ge  2^{k}
\end{equation}
and thus that $\p(\beta)\le p<2$. Now let us prove \eqref{tropp}. The assumption $\beta>\beta_2$ implies that $\bbE[( W^\beta_{n})(x)^2]$ increases exponentially in $n$ but one cannot deduce \eqref{tropp} directly from it: since $X$ has unbounded jumps the  mass of $W^\beta_{n}$ could, a priori, be spread out over exponentially many points.

\smallskip

To circumvent this issue we
obtain, via an ad-hoc argument,  a bound for the second moment of the partition function restricted to trajectories ending in a ball whose radius is a power of $n$.

 \medskip
Note that $\beta>\beta_2$ implies  (cf. \eqref{defpbeta2}) that $ \chi \pi>1,$
where $\chi=\chi(\beta)=e^{\gl(2\beta)-2\gl(\beta)}$ and $\pi\coloneqq P^{\otimes 2}(\tau<\infty)$ for $\tau\coloneqq \inf\{n\geq 1\colon X^{(1)}_n=X^{(2)}_n\}$. By continuity, we can find $m\in\N$ and $M>1$ such that $\chi\pi_{m,M}>1$, where
\begin{align*}
\pi_{m,M}=P^{\otimes 2}\left(\tau\leq m,|X_\tau^{(1)}|\leq M\right).
\end{align*}
We set
\begin{equation}\begin{split}
L_n&= L_n(X^{(1)},X^{(2)}):= \sum^n_{k=1}\ind_{\{X^{(1)}_k=X^{(2)}_k\}}\\
\widetilde W_n&:=E\left[e^{\beta\sum_{k=1}^{n}\beta\omega_{k,X_k}-n\lambda(\beta)}\1_{\{|X_n|\leq 2n^{1+2/\eta}\}}\right].
\end{split}
\end{equation}
Taking $l$ such that $Ml\le (ml)^{1+2/\eta} $, we have
\begin{align*}
\bbE\left[\widetilde W_{ml}^2\right]&=E^{\otimes 2} \left[\chi^{L_{ml}}\1_{\{\max(|X^{(1)}_{ml}|,|X^{(2)}_{ml}|)\leq 2(ml)^{1+2/\eta}\}}\right]\\
&\geq \sum_{n_1,...,n_l\leq m}\chi^l\left(\prod_{i=1}^lP^{\otimes 2}\left(\tau=n_i,|X^{(1)}_{n_i}|\leq M\right)\right)P\left(|X_{ml-\sum_{i=1}^{l}n_i}|\leq (ml)^{1+2/\eta}\right)^2\\
&\geq \left(\chi\pi_{m,M}\right)^l
P\left(|X_1|\le (ml)^{2/\eta}\right)^{2ml}\geq \frac 12\left(\chi\pi_{m,M}\right)^l.
\end{align*}
Above, the first inequality is obtained by restricting the expectation to trajectories having at least $l$ intersections and such that the intervals between successive intersection times are smaller than $m$. The last inequality is valid for $l$ sufficiently large as a consequence of \eqref{defeta}.
Thus, again for $l$ large enough, we have
\begin{align*}
\max_{|x|\leq 2(ml)^{1+2/\eta}}\bbE\left[ ( \hat W^{\beta}_m(x))^2 \right]\ge \frac{1}{\left(1+4(ml)^{1+2/\eta}\right)^{2d}} \bbE\left[ \widetilde W_{ml}^2 \right]\ge 3.
\tag*{\qed}
\end{align*}

\subsection{Proof of Corollary~\ref{tailcor}}
For part $(i)$, since $\p$ is decreasing, it is enough to show that for any $\eps>0$ there exists $\beta'<\beta$ such that $\p(\beta')\leq \p(\beta)+\eps$, that is 
\begin{equation}\label{lesupp}
 \sup_n\E[(W_n^{\beta'})^{\p(\beta)+\eps}]=\infty.
\end{equation}
 For this purpose, let $u_0$ be large enough that $u_0^{\eps}\geq \frac{8}{c_1}$, where $c_1$ is such that $\P(\hat W_{\infty}^{\beta,\ast}\ge u)\geq c_1u^{-\p}$ holds for all $u>1$, using \eqref{asympsymp}. We then choose $T_0$ large enough so that
\begin{equation}\label{contii}
\bbP\Big(\sup_{x\in\bbZ^d,t\leq T_0}\hat W_t^\beta(x)>u_0\Big)\geq \frac 12\bbP\Big(\sup_{x\in\bbZ^d,t\in\bbN}\hat W_t^\beta(x)>u_0\Big)\geq \frac {c_1}2u_0^{-\p}.
\end{equation}
 Since the variable $\sup_{x\in\bbZ^d,t\leq T_0}\hat W_t^\beta(x)$ is continuous in $\beta$, using Portmanteau's Theorem,  we can find $\beta'<\beta$ such that
\begin{align}\label{g1}
\bbP\Big(\sup_{x\in\bbZ^d,t\leq T_0}\hat W_t^{\beta'}(x)>u_0\Big)\geq \frac {c_1}4u_0^{-\p}.
\end{align}
Using \eqref{inductive}, we have
\begin{align}\label{g2}
\bbP\Big(W_{kT_0}^{\beta'}\geq u_0^k\Big)\ge  \bbP\Big(\sup_{x\in \bbZ^d}\hat W_{kT_0}^{\beta'}(x)\geq u_0^k\Big)
\ge \bbP\Big(\sup_{x\in\bbZ^d,t\leq T_0}\hat W_t^{\beta'}(x)>u_0\Big)^k
\end{align}
and thus combining \eqref{g1} and \eqref{g2} we obtain 
\begin{align}\label{g3}
\bbE\Big[(W_{kT_0}^{\beta'})^{\p+\eps}\Big]\geq (u_0)^{(\p+\eps)k}\bbP\Big(W_{kT_0}^{\beta'}\geq u_0^k\Big) \ge   \Big(\frac{c_1}{4} u_0^{\eps}\Big)^k \ge 2^k
\end{align}
where the last inequality comes from our choice for $u_0$. This proves \eqref{lesupp}.  For part $(ii)$ we repeat the above argument with $\beta'$ replaced by $\beta$, which gives
\begin{align*}
\limsup_{n\to\infty}\frac 1n\log \E[(W_n^\beta)^{\p+\eps}]\geq \lim_{k\to \infty} \frac{1}{kT_0} \log \left[ \Big(\frac{c_1}{2} u_0^{\eps}\Big)^k\right]\ge \frac{\log 4}{T_0}>0.
\end{align*}
\qed

\begin{rem}\label{otherrem}
In order to justify the validity of the above argument in the framework of Proposition~\ref{gen}, one needs to justify
the continuity in $\beta$ of  $\sup_{x\in\bbZ^d,t\leq T_0}\hat W_t^\beta(x)$ (when the increments of the walks are unbounded the supremum is taken over an infinite set). To bypass this, one may consider a supremum over a finite set by replacing \eqref{contii} by
\begin{align*}
\bbP\Big(\sup_{x\in\intp{-L_0,L_0}^d,t\leq T_0}\hat W_t^\beta(x)>u_0\Big)\geq \frac 12\bbP\Big(\sup_{x\in\bbZ^d,t\in\bbN}\hat W_t^\beta(x)>u_0\Big)\geq \frac {c_1}2u_0^{-\p},
\end{align*}
for suitably large $T_0$ and $L_0$, and modify \eqref{g2} accordingly.
\end{rem}

\subsection{Proof of Corollary~\ref{corgrowth}}

It it easy to see that $\bbE\left[(W_n^\beta)^{\p(\beta)}\right]$ diverges: by Corollary~\ref{pcor}(i), we have $\p(\beta)>1$ and thus Doob's inequality implies $\liminf_{n\to\infty}\bbE\left[(W_n^\beta)^{\p(\beta)}\right]\geq C \bbE\left[(W^{\beta,*}_\infty)^{\p(\beta)}\right]$. The second expression is infinite due to Theorem~\ref{thm:tail}.

\smallskip For \eqref{sublin}, we observe that
\begin{equation}\label{formula}
  \bbE\left[ (W^\beta_n)^{p^*} \right]= \int^\infty_0 p^* u^{p^*-1} \bbP\left(  W^\beta_n\ge u \right) \dd u
  \le 1+  \int^\infty_1 p^* u^{p^*-1} \bbP\left(  W^\beta_n\ge u \right) \dd u.
 \end{equation}
Now Theorem~\ref{thm:tail} implies that, for any $u\ge 1$ and some constant $C>0$,
\begin{equation}\label{smallu}
 u^{p^*-1} \bbP\left(  W^\beta_n\ge u \right) \le C u^{-1}.
 \end{equation}
We are going to show that there exists a constant $K>0$ such that for every $n$ we have
\begin{equation}\label{fromK}
\forall u \ge \exp(K n), \quad  u^{p^*-1} \bbP\left( W^\beta_n\ge u \right) \le  u^{-2}.
\end{equation}
Combining \eqref{formula}, \eqref{smallu} and \eqref{fromK} we obtain that 
\begin{equation}
   \bbE\left[ (W^\beta_n)^{p^*} \right] 
   \le 1+ Cp^* \int^{e^{Kn}}_1  u^{-1} \dd u +  p^*\int^\infty_{e^{Kn}}  u^{-2} \dd u\le 1+p^*+ C K p^* n.
\end{equation}
Finally, to prove \eqref{fromK}, we simply observe that
\begin{equation}
 \bbP\left(  W^\beta_n\ge u  \right)\le u^{-p} \bbE\left[ (W^{\beta}_n)^p \right]
 \le u^{-p} e^{n\left( \gl(\beta p)- p\gl(\beta)\right)}
 \end{equation}
 and conclude by applying the above to $p\coloneqq p^*+2$, which implies \eqref{fromK} for $K=\gl(\beta p)- p\gl(\beta)$.
\qed

\section*{Acknowledgements}
This work was partially realized as H.L.\ was visiting  the Research Institute for Mathematical Sciences, an International Joint Usage/Research Center located in Kyoto University and he  acknowledges kind hospitality and support. 
This research was supported by JSPS Grant-in-Aid for Scientific Research 23K12984 and 21K03286. H.L.\ also acknowledges the support of a productivity grand from CNQq and of a CNE grant from FAPERj.
 We are grateful to two anonymous referees for their careful reading of this manuscript and for their valuable feedback.

\section*{Competing interests}
\noindent The authors have no relevant financial or non-financial interests to disclose.

\bibliographystyle{alpha}
\bibliography{ref.bib}

\end{document}